\documentclass[final]{siamltex}

\usepackage[english]{babel}
\usepackage{amsmath}
\usepackage{amssymb}
\usepackage{graphicx}
\usepackage{tikz}
\usepackage{psfrag}
\usepackage[utf8]{inputenc}

\newcommand\R{\mathbb{R}}

\newcommand\N{\mathbb{N}}

\newtheorem{remark}{Remark}
\newcommand{\norm}[1]{\| #1 \|}

\newtheorem{algorithm}{Algorithm}

\title{Reduced basis methods for pricing options with the Black-Scholes 
 and Heston model}
\author{O. Burkovska\thanks{
Fakult\"at Mathematik, Technische Universit\"at M\"unchen, Germany; {\tt burkovsk@ma.tum.de}. The author acknowledges the support from the International Research Training Group IGDK1754, funded by the German Research Foundation (DFG).} 
\and
B. Haasdonk
\thanks{
Institut für Angewandte Analysis und Numerische Simulation, Universit\"at Stuttgart, Germany; {\tt haasdonk@mathematik.uni-stuttgart.de}. The author acknowledges the
German Research Foundation (DFG) for financial support of the project within the Cluster of Excellence in Simulation Technology (EXC 310/1) at the University of Stuttgart and the Baden-Württemberg Stiftung gGmbH.}
\and J. Salomon
\thanks{CEREMADE, Universit\'e Paris-Dauphine, France; {\tt julien.salomon@dauphine.fr}} 
\and B. Wohlmuth\footnotemark[1]
\thanks{
Fakult\"at Mathematik, Technische Universit\"at M\"unchen, Germany; {\tt wohlmuth@ma.tum.de}} 
}

\begin{document}
%
%
\maketitle
\begin{abstract}
In this paper, we present a reduced basis method for pricing European and 
American options based on the Black-Scholes and Heston model. To tackle 
each model numerically, we formulate the problem in terms of a
time dependent variational equality or inequality. We apply a suitable 
reduced basis approach for both types of options. 
The characteristic ingredients used in the method are a combined POD-Greedy
and Angle-Greedy procedure for the construction of the primal and dual 
reduced spaces. Analytically, we prove the reproduction property of the 
reduced scheme and derive a posteriori error estimators. Numerical examples 
are provided, illustrating the approximation quality and convergence of our 
approach for the different option pricing models. Also, we investigate the 
reliability and effectivity of the error estimators.
\end{abstract}
\section{Introduction}

We consider the problem of European and American option pricing and refer to~\cite{achdou,hilber,seydel} 
and the references therein for an introduction into computational methods for option pricing.
While European options can be modeled by a parabolic partial differential
equation, American options result in additional inequality constraints. Different models can be 
used to price European and American options. The simplest ones, e.g., the Black-Scholes 
model~\cite{black_scholes}, assume that the volatility is constant. However, in most of the cases, 
the real market violates this assumption due to its stochastic nature. Thus, alternative models, 
which try to capture this phenomenon are frequently used, e.g., the Heston stochastic volatility 
model~\cite{heston}. 

Another difficulty which arises with solving partial differential equations (PDEs) 
for option pricing, in particular for pricing American options, is that for most of the models no closed form solution exists. Thus one has to develop appropriate numerical methods. The common methods to 
solve pricing equations with the Heston model are finite differences, cf.~\cite{Galiotos08,hout2010,ikonen2008,ikonen2009} and finite elements, cf.~\cite{feng2011,kunoth2012,winkler2001, Zvan98}.  
We refer to~\cite{wohlmuth}  for a possible numerical treatment of basket options with the Black-Scholes model by primal-dual finite elements and to~\cite{geiger,glowinski,kinderlehrer} for an abstract framework on the theory of constrained variational
problems.

We are interested in providing fast numerical algorithms to accurately solve
the variational
equality and inequality systems associated with European and American call and put options 
for a large variety of different parameter values such as, e.g., interest rate, dividend and correlation. 
Reduced basis (RB) methods are an appropriate means
for standard parametrized parabolic partial differential equations,
cf. \cite{buffa,grepl2005reduced,haasdonk08,rozza,veroy}
and the references therein. These techniques are based on
low-dimensional approximation spaces, that are constructed
by greedy procedures. Convergence behavior is known in some cases
\cite{buffa, Ha13}.
The computational advantage of RB-methods over standard discretization 
methods 
is obtained by its possible offline/online decomposition: First, a typically 
expensive offline-phase involving the computation of the reduced spaces is 
performed. This phase only needs to be precomputed once. Then, the online phase allows an extremely fast computation of the RB
solutions for many new parameters 
as only low dimensional systems need to be solved. Recently, we adopted the RB methodology to
constrained stationary elliptic problems \cite{HSW12}, which
we extend here to the instationary case.

We refer to the recent contribution~\cite{pironneau} for a
tailored RB approach in option pricing with diffusion and jump-diffusion models, which later was 
generalized to basket options~\cite{pironneau11} and, in fact, was shown to be a variant of 
a Proper Orthogonal Decomposition method (POD)~\cite{pironneau12}. The application of the 
reduced basis method can be also extended to the calibration of option pricing models, 
e.g.,~\cite{pironneau09,sachs2008}. In contrast to our setting, no inequality constraints are
taken into account there. 
Further work relevant for RB-methods and variational inequalities comprises
\cite{UG13}, which addresses a time-space formulation of the problem and corresponding 
analysis. Also, recently, ongoing work has been presented in~\cite{VeroyParisWorkshop}, 
which alternatively treats the inequality constraints both by a primal-dual technique 
and a penalization approach.

One main challenge in our problem setting is the construction of a suitable low dimensional 
approximation of the dual cone required for the 
treatment of the constraints. In this work, we present an algorithm to overcome this difficulty 
which is based on the greedy procedure and tries to capture as much ``volume'' as possible in 
the construction of the dual cone. This is obtained by iteratively selecting snapshots maximizing the angle to the current space. As second main ingredient for treatment of 
additional inequality constraints, we provide analytical results, in particular 
a posteriori error control. 

Let us briefly relate the current presentation to our previous works. 
In contrast to \cite{HSW12}, 
which dealt with RB methods for stationary variational inequalities,
we treat here instationary problems. In this way, we apply sharper a posteriori error 
estimator strategies, which can straightforwardly be applied to the stationary case and 
gives improvements over \cite{HSW12}.
In \cite{HSW12b} we presented an RB procedure for American option pricing with a Black-Scholes 
model and gave first simple examples. The current presentation considerably extends this by 
including the Heston model and both European and American options and providing an analysis 
for the corresponding RB scheme. 
 

The content of this paper is structured as follows: In Section~\ref{sec:var_form}, we present both a strong and a variational formulation of pricing of European and American options with the Black-Scholes and Heston model. In Section~\ref{sec:rb}, a reduced basis method is introduced together with the construction of the primal and dual reduced basis spaces. Section~\ref{sec:error_analysis} contains the a posteriori error analysis induced by the method and derived from equality and inequality residuals.
The implementational aspects of the reduced basis method together with the construction algorithms for reduced basis spaces are presented in Section~\ref{sec:implementation}. Numerical results, given in Section~\ref{sec:results}, illustrate the performance of the method for pricing options in one dimensional Black-Scholes and two dimensional Heston models.

\section{Variational formulations for European and American Options}\label{sec:var_form}
\subsection{Option Pricing Models}
In this section, we give a brief introduction to the theory of option pricing. An option is a contract which permits its owner the right to buy or sell an underlying asset (a stock, or a parcel of shares)  at a prespecified fixed strike price $K\geq 0$  before or at a certain  time $T\geq 0$, called maturity. There are two basic types of an option: a {\em call} option which gives a holder a right to buy and a {\em put} option which allows an owner to sell an underlying asset. Also one distinguishes between European options, where exercise is only permitted at maturity $T$ and  American options which can be exercised at any time before an expiration time $T$. We will denote the price of the underlying asset by $S=S_\tau\in\mathbb{R}_+$, where $\tau\geq 0$ is the time to maturity $T$.

In standard option pricing models, e.g., the Black-Scholes model~\cite{achdou,black_scholes,hull93}, a price of the underlying asset $S_\tau$ follows a stochastic process,  governed by the following stochastic differential equation
\begin{eqnarray}
dS_\tau = \iota S_\tau dt+\sigma_\tau S_\tau dW_\tau,
\end{eqnarray}
with a Wiener process $W_\tau$, a drift $\iota$ and a volatility $\sigma_\tau > 0$. One of the main limitations of these models is the assumption that the volatility of the return on the underlying asset is constant $\sigma_\tau = \sigma$, while in financial markets, the volatility  is not a constant, but a stochastic variable. The Heston stochastic volatility model~\cite{heston} takes into account the randomness of the volatility and is based on the following stock price and variance dynamics
\begin{eqnarray}
dS_\tau &=& \iota S_\tau dt+\sqrt{v_\tau}S_\tau dW_\tau^1,\\
\label{eq:varprocess}
dv_\tau &=& \kappa(\gamma-v_\tau)d\tau + \xi\sqrt{v_\tau}dW_\tau^2,
\end{eqnarray}
where $v_\tau=\sigma^2$  follows a square root process (known as a Cox-Ingersoll-Ross (CIR) process) with the mean variance $\gamma>0$, rate of mean reversion $\kappa>0$ and so called volatility of volatility $\xi>0$. The Wiener processes $W_\tau^1$ and $W_\tau^2$ are correlated with the correlation parameter $\rho\in [-1,1]$. 

With the use of Ito's formula~\cite{hull93}, each model can be formulated in terms of a partial differential equation (PDE). For purposes of brevity, we omit the derivation of the equations. The reader is referred to~\cite{achdou,heston,hilber} for further details.

We define the spatial differential operators corresponding to the Black-Scholes and Heston model as follows
\begin{align}
\mathcal{L}^{BS}P:&=\frac{1}{2}\sigma^{2}S^2\partial_{SS}P+rS\partial_{S}P-(r-q)P,\label{eq:BS}\\
\mathcal{L}^{H}P:&=\frac{1}{2}\xi^{2}v\partial_{vv}P+\rho\xi vS\partial_{vS}P+\frac{1}{2}vS^{2}\partial_{SS}P+\kappa(\gamma-v)\partial_{v}P+rS\partial_{S}P-rP,\label{eq:Heston}
\end{align}
where  $r$ is the interest rate and $q$ is the dividend payment.
Then the value of an option $P(\tau,S)$ in the Black-Scholes model and $P(\tau,v,S)$ in the Heston model paying $P_0(S)$ at maturity time $T$ must satisfy the following partial differential equation 
\begin{eqnarray}
\partial_\tau P+\mathcal{L}P = 0,\label{eq:Heston1}
\end{eqnarray}
where the value $P_0(S)$ is called a payoff function and $\mathcal{L}:=\mathcal{L}^{s}$, $s=\{BS,H\}$ for the Black-Scholes and the Heston model, respectively. 

\begin{remark}\label{remark:BS}
 Assuming a constant volatility in the Heston model and no dividend payment in the Black-Scholes model, $q=0$, we have $v=\gamma$ and $\xi=0$, and the Heston equation reduces to the Black-Scholes equation with the  constant volatility $\sigma=\sqrt{v}$.
\end{remark}

For further derivation of the weak formulation of the problem (\ref{eq:Heston1}), we introduce the following
notation.
Let $\Omega\subset\mathbb{R}^d$, $d=1,2$, be a bounded open domain with Lipschitz continuous boundary $\partial\Omega$. We consider $d=1$ for the Black-Scholes model and $d=2$ for the Heston model. 
Note that this is not a limitation for the RB-approach that we present. 
Higher dimensional problems can readily be treated, as soon as suitable solvers 
for the discretized PDE are available.   
We introduce the following functional space
\begin{align}
V:=\left\{ \phi\in H^1(\Omega):  \ \ \phi=0 \ \text{ on  }\partial \Omega_D \right\}, &
\end{align}
where $\partial\Omega_D$ denotes a Dirichlet portion of the boundary $\partial \Omega$. Further $\langle\cdot, \cdot \rangle_{V}$ and $\norm{\cdot}_V$ denote 
the inner product and norm of $V$, similar for other spaces. In the error analysis, we make use of a norm bound, i.e. we denote by $C_\Omega$ a constant that 
satisfies, for any $v\in V$: 
\begin{equation}\label{ineq:normequi}
 \|v\|_{L^2(\Omega)}\leq C_{\Omega} \|v\|_V.
\end{equation}
In the experiments, we use $H^1$ or weighted $H^1$ norms for $V$, but for the
sake of generality of our analysis, we introduce this generic constant that includes other frameworks, e.g., $H^1_0$ norm
for which $C_\Omega$ is the Friedrichs-Poincaré constant.

We define a backward time variable $t:=T-\tau$ which we will use throughout the paper. This transforms the PDE (\ref{eq:Heston1}) into a standard forward evolution problem. We denote by $(\cdot)_+:=\max(0,\cdot)$, by ``$\circ$" the Hadamard product and  
introduce a parameter $\mu\in\mathcal{P}\subset\mathbb{R}^p$, $p=3,5$, which parametrizes (\ref{eq:Heston1}). We set $\mu:=(\sigma,q,r)$ for the Black-Scholes model and $\mu:=(\xi,\rho,\gamma,\kappa,r)$ for the Heston model.

\subsection{Pricing of European options}\label{sec:EO}

We start our consideration with the simplest case of pricing the European type of options. Since puts and calls of the European option can be easily interchanged via a put-call parity relation~\cite[p. 163]{hull93}, it suffices for us to consider, e.g., only call options. In addition, there exists a semi-closed analytical solution for these options in the Black-Scholes and Heston models. However, while the Black-Scholes formula~\cite{black_scholes} provides almost the exact value of the European option, the Heston semi-analytical formula~\cite{heston} requires some numerical techniques to approximate the integral. Thus, it is more interesting to consider the case of the Heston model. Then the value of a European call option $P(t,v,S)$, $(v,S)\in\mathbb{R}^2_+$ (the relation to $\Omega$ being established soon) satisfies the following linear equation
\begin{align}
\partial_tP-\mathcal{L}^{H}P=0, \ \ \  \ \ t\in (0,T],
\label{eq:EO1}
\end{align}
subject to initial and boundary conditions
\begin{align}
\nonumber &P(0,v,S) =(S-K)_{+}, &&\lim_{S\rightarrow 0}P(t,v,S) &&= 0,\\
\nonumber &\lim_{S\rightarrow+\infty}P_{S}(t,v,S) = 1, &&\lim_{v\rightarrow+\infty}P(t,v,S) && = S,\\
&rP(t,0,S) = rSP_{S}(t,0,S)+\kappa\gamma P_{v}(t,0,S)+P_{t}(t,0,S).
\end{align}
The description of the boundary conditions can be interpreted as follows: For a large stock price $S$, we use a Neumann boundary condition which establishes a linear growth of an option price. When the volatility $v$ is 0, we cannot impose any boundary conditions and assume that the equation (\ref{eq:EO1}) is satisfied  on the line $v=0$. When the stock price is worthless $S=0$, it is  natural to assume that the value of the call is also worthless. The option price is increasing with the volatility but it remains bounded by the stock price, hence when the volatility $v$ is large, we assume the value of the option tends to reach the value of the stock price $S$. Note, this is not the only way to prescribe boundary conditions for the problem, other types of boundary conditions can be found, e.g., in ~\cite{duffy06,Galiotos08,Zvan98}.
\begin{remark}
The variance process (\ref{eq:varprocess}) is strictly positive if the condition on the parameters $\xi^2<2\kappa\gamma$ is satisfied, which is often referred as the Feller condition, see, e.g.,~\cite{janek2011fx}. This condition plays a crucial role in the calibration process of the Heston model, and it is uncommon that the parameters violate it. Thus we restrict ourselves to the choice of model parameters, such that the Feller condition is fulfilled. 
\end{remark}

Since the operator $\mathcal{L}^H$ in (\ref{eq:Heston}) is a degenerate parabolic differential operator, the standard way to eliminate the variable coefficient $S$ is to perform the log-transformation of $S$ by introducing a new variable $x:=\log\left(\frac{S}{K}\right)$. Then we are looking for the solution $w(t,v,x):=P(t,v,\log(\frac{S}{K}))$, with the initial condition $w^0(x):=w(0,v,x)=(Ke^x-K)_+$  which satisfies the transformed Heston equation 
\begin{eqnarray}
\label{eq:Heston_tr}
\partial_tw -L^Hw=0,
\end{eqnarray} 
for all $(v,x)\in\mathbb{R}_+\times (-\infty,+\infty)$. The operator $L^H$ corresponds to the operator $\mathcal{L}^H$ in (\ref{eq:Heston}) with respect to a change of variables and is defined as follows
\begin{eqnarray}\label{eq:operatorH}
 L^H w:=\nabla\cdot A\nabla w-b\cdot\nabla w-rw,
\end{eqnarray}
with
\begin{eqnarray}\label{eq:heston_coef}
A:=\frac{1}{2}v\begin{bmatrix}
\xi^2 &\rho\xi \\ \rho\xi &1
\end{bmatrix}, \ \ \ \ \ \ \ \ \ \ \ \ 
b:=\begin{bmatrix}
-\kappa(\gamma-v)+\frac{1}{2}\xi^2 \\ -r+\frac{1}{2}v+\frac{1}{2}\xi\rho
\end{bmatrix}
\end{eqnarray}
and first oder spatial differential operator $\nabla :=\left(\partial_v ,\partial_x \right)^T$.

To perform a numerical simulation, we localize the problem (\ref{eq:Heston_tr}) to a bounded computational domain $\Omega=(v_{\min}, v_{\max})\times(x_{\min}, x_{\max})\subset\mathbb{R}^2$ with the variance  $v_{\min}>0$. Then the boundary conditions transform into
\begin{eqnarray*}
\Gamma_1: \ &v=v_{min}\ \ \ \ &w(t,v_{min},x)=Ke^{x}\Phi(d_+)-Ke^{-rt}\Phi(d_-),\\
\Gamma_2: \ &v=v_{max}\ \ \ \ &w(t,v_{max},x)=Ke^{x},\\
\Gamma_3: \ &x=x_{min}\ \ \ \ &w=\lambda w(t,v_{max},x_{min})+(1-\lambda)w(t,v_{min},x_{min}), \\
& &\lambda=\frac{v-v_{min}}{v_{max}-v_{min}},\\
\Gamma_4: \ &x=x_{max}\ \ \ \ &A\frac{\partial w}{\partial\overline{n}}(t,v,x_{max})=\frac{1}{2}vKe^{x},
\end{eqnarray*}
with $\overline{n}$ being an outward normal vector at the boundary,
$\sigma=\sqrt{v}$ and 
$d_{\pm}$ and a cumulative distribution function  $\Phi(x)$ defined as
$$d_{\pm}=\frac{x+(r\pm\frac{\sigma^2}{2})t}{\sigma\sqrt{t}}, \quad
\Phi(x)=\frac{1}{\sqrt{2\pi}}\int_{-\infty}^x e^{-\frac{z^2}{2}}dz.$$ 
 Using the test function $\phi\in V$ yields the following weak formulation of (\ref{eq:Heston_tr})
\begin{eqnarray}
 \langle\partial_t w,\phi\rangle_{L^2(\Omega)} + a^H(w,\phi;\mu)=f^E(\phi;\mu),\ \ \ \forall \phi\in V,
\end{eqnarray}
where 
\begin{eqnarray}
 a^H(w,\phi;\mu)&:=&\langle A\nabla w,\nabla\phi\rangle_{L^2(\Omega)}+\langle b\cdot\nabla w,\phi\rangle_{L^2(\Omega)}+r\langle w,\phi\rangle_{L^2(\Omega)},\label{eq:bilHeston}\\
f^E(\phi;\mu)&:=&\langle A\frac{\partial w}{\partial\overline{n}},\phi\rangle_{L^2(\Gamma_4)}.
\end{eqnarray}
Various methods can be applied to solve this problem numerically, e.g.,~\cite{hout2010, janek2011fx, winkler2001}. 

\subsection{Pricing of American options}\label{sec:AO}
 We extend our considerations to American options based on the Black-Scholes and Heston model, for which, in general, a closed form solution does not exist. Since American calls are equal to European calls on non-dividend paying stocks~\cite[p. 159]{hull93}, we will focus only on American put options in both models. 
Following the same arguments as in Section \ref{sec:EO}, we consider the price of the American put in the  log-transformed variable $w(x):=P(Ke^x)$ with a log-transformed payoff $w^0(x):=(K-Ke^x)_+$ which solves the following problem 
\begin{eqnarray}
\partial_tw-Lw\geq 0, \ \ \ \ \  w-w^0&\geq& 0,\label{eg:AO1_trans}\\
\left(\partial_tw-Lw\right)\cdot\left(w-w^0\right)&=&0\label{eq:AO2_trans}.
\end{eqnarray}
The operator $L:=L^s$, $s=\{BS,H\}$ for the Black-Scholes or Heston model, where
\begin{eqnarray}\label{operatorBS}
 L^{BS}w:=\frac{1}{2}\sigma^2\partial^2_{xx}w+(r-q)\partial_xw-rw. 
\end{eqnarray}
The boundary conditions for the Heston model are set to $w=w^0(x)$ on $\partial\Omega_D:=\Gamma_1\cup\Gamma_3\cup\Gamma_4$ and $\frac{\partial w}{\partial \overline{n}}=0$ on $\Gamma_2$. For the Black-Scholes model we 
define $\partial\Omega_D:=\{x_{\min}, x_{\max} \}$ 
and on this boundary we prescribe  $w=w^0(x)$. 

Our aim is to reformulate the system (\ref{eg:AO1_trans})--(\ref{eq:AO2_trans}) in a variational saddle point form~\cite{kikuchi}. Define  $W:=V'$ to be the dual space of $V$ and 
$M\subset W$ to be a dual cone. 
For all $\eta\in W$, $v\in V$ introduce a duality pairing 
$b:W\times V\to \mathbb{R}$, $b(\eta,v)=\langle\eta,v\rangle_{V^\prime,V}$ and  $\widetilde{g}(\eta;\mu):=b(\eta,w^0)$. 
The bilinear form of the Black-Scholes equation reads
\begin{eqnarray}
 a^{BS}(w,\phi;\mu)&:=&\frac{1}{2}\sigma^2\langle\partial_xw,\partial_x\phi \rangle_{L^2(\Omega)}-(r-q)\langle\partial_xw,\phi\rangle_{L^2(\Omega)}+r\langle w,\phi\rangle_{L^2(\Omega)}. \quad \label{eq:bilBS}
\end{eqnarray}
To treat the problem numerically, we use a $\theta$-scheme for the discretization in time 
 and conforming piecewise linear finite elements 
for the discretization in spatial direction. We divide $(0,T]$ into $L$ subintervals  of equal length $\Delta t:=\frac{T}{L}$ and define $w^n:=w(t^n,v,x)\in H^1(\Omega)$, $t^n:=n\Delta t$, $0< n\leq L$.  In order to ensure Dirichlet boundary conditions, we set $w^n=u^n+u_g^n$, where $u^n:=u(t^n,\cdot,\cdot)\in V$ solves~\eqref{eg:AO1_trans}--\eqref{eq:AO2_trans} with homogeneous Dirichlet boundary conditions and $u_g^n:=u_g(t^n,\cdot,\cdot)\in H^1(\Omega)$ is a Dirichlet lift function, which extends non-homogeneous boundary conditions to the interior of the domain. For $0<n\leq L-1$ we introduce the operators
\begin{align}
 f^n(\phi;\mu)&:=-\left\langle\frac{u_g^{n+1}-u_g^n}{\Delta t},\phi\right\rangle_{L^2(\Omega)}-a(\theta u^{n+1}_g+(1-\theta)u_g^n,\phi;\mu),\\
 g^n(\eta-\lambda^{n+1};\mu)&:=\widetilde{g}(\eta-\lambda^{n+1};\mu)-b(\eta-\lambda^{n+1},u_g^{n+1}),
\end{align}
and the discrete problem in a saddle point form reads: 
\begin{definition}[Detailed problem]
 For $\mu\in\mathcal{P}$ and given initial data $u^0 \in V$ find 
$(u^{n+1}(\mu),\lambda^{n+1}(\mu))\in V\times M$ for $0<n\leq L-1$ and $\phi \in V,\eta \in M$ 
satisfying 
\begin{align}
&\left\langle\frac{u^{n+1}-u^n}{\Delta t},\phi\right\rangle_{L^2(\Omega)}+a(\theta u^{n+1}+(1-\theta)u^n,\phi;\mu) -b(\lambda^{n+1},\phi)  =  f^n(\phi;\mu),
\label{eq:AOdisc1}\\
&b(\eta-\lambda^{n+1},u^{n+1})  \geq g^n(\eta-\lambda^{n+1};\mu). \label{eq:AOdisc2} 
\end{align}
\end{definition}
The bilinear form $a(\cdot,\cdot;\mu):=a^s(\cdot,\cdot;\mu)$, $s=\{BS,H\}$ is specified for each model in~\eqref{eq:bilBS} or~\eqref{eq:bilHeston}. Here and in the following, we frequently omit 
the argument $\mu$ whenever the parameter value is clear from the context. 
The problem (\ref{eq:AOdisc1})--(\ref{eq:AOdisc2}) can be considered as a model independent formulation for pricing American put options. Moreover, the European call option is also enclosed in this formulation by the exchange of $f^n(\cdot;\mu)$ with $f^n(\cdot;\mu)+f^E(\cdot;\mu)$, initial conditions and omitting $b(\cdot,\cdot)$ and $g^n(\cdot,\mu)$ terms. Therefore, in the further discussions of the implementation aspects and analysis, we will  focus only on the general American put option case, and we will not distinguish the models which are used to price the option. However in Section \ref{sec:results}, we present and compare numerical results for European and American options in both models. 

By a generalized Lax-Milgram argument, a problem of type~(\ref{eq:AOdisc1})--(\ref{eq:AOdisc2}) is well-posed if the bilinear form $a(\cdot,\cdot,\mu)$ is continuous and coercive, $f^n(\cdot,\mu)$, $g^n(\cdot;\mu)$ are linear and continuous and $b(\cdot,\cdot)$ is inf-sup stable. In particular, for the Heston model, if $v\geq v_{\min}>0$ the matrix $A$ in (\ref{eq:heston_coef}) is positive definite on $\overline{\Omega}$ and under the suitable relation on the coefficients, we obtain the coercivity and continuity of $a(\cdot,\cdot;\mu)$. This issue for the European call option was studied in great detail in~\cite{winkler2001}. The well-posedness of the problem in the Black-Scholes settings can be found, e.g., in~\cite[Ch. 6]{achdou}. The coercivity, continuity, and inf-sup constants are defined as follows
\begin{align}
 \label{eq:infsup}
 \alpha_a(\mu)&:=\inf_{u\in V}\frac{a(u,u;\mu)}{\|u\|_V^2}>0,
 \ \ \ \ \gamma_a(\mu):=\sup_{u\in V}\sup_{v\in V}\frac{a(u,v;\mu)}{\|u\|_V\|v\|_V}<\infty,  \ \ \ \forall\mu\in\mathcal{P},\\
 \beta &:=\inf_{\eta\in W}\sup_{v\in V}\frac{b(\eta,v)}{\|\eta\|_W\|v\|_V}>0.
\end{align}
Also, for our choice of a dual space and cone we assume
\begin{eqnarray}
\label{eq:star}
 \langle\eta,\eta^\prime\rangle_W\geq 0, \quad\forall\eta,\eta^\prime\in M.
\end{eqnarray}

\section{Reduced basis setting}\label{sec:rb}
In this section, we provide the RB-scheme for the variational inequality 
problem and present the main analytical results.
\subsection{Reduced basis discretization}
Standard finite element approaches do not exploit the structure of the
solution manifold under parameter variation and for a given parameter value, 
a high dimensional system has to be solved. In what follows, we introduce a
specific Galerkin approximation of the solution, based on the reduced
basis method. 
The first step of the reduced basis method mainly consists in computing parametric  
solutions in low dimensional subspaces of $V$ and
$W$, defined in Section~\ref{sec:var_form}, that are generated 
with particular solutions, the so-called {\it snapshots}, of our problem.

Let us explain the corresponding formulation in more detail.
For $N\in\N$, consider a finite subset ${\cal
P}_N:=\left\{\mu_1,\ldots,\mu_N \right\}\subset {\cal P}$ with
$\mu_i\neq\mu_j,\ \forall i\neq j$. The reduced spaces
$V_N$ and $W_N$ are defined by $V_N:={\rm
  span}\left\{\psi_1,\ldots,\psi_{N_V}\right\}$ and $W_N:={\rm
  span}\left\{\xi_1,\ldots,\xi_{N_W}\right\}$ 
where $\psi_i \in V$ and $\xi_i \in M$ are suitably constructed from
the large set of snapshot solutions $u^n(\mu_i)$, $i= 1, \ldots, N
$, $n = 0, \ldots, L$  and $\lambda^n(\mu_i)$, $i= 1, \ldots, N
$, $n = 1, \ldots, L$
and the reduced dimensions 
$N_V,N_W$ are preferably small.
Both
families $\Psi_N=(\psi_j)_{j=1,\ldots,N_V}$ and $\Xi_N=(\xi_j)_{j=1,\ldots,N_W}$ are supposed to be composed of
linearly independent functions, hence are so-called {\it reduced bases}.
Numerical algorithms to build these two sets will be presented in Section~\ref{sec:RBcomp}. 
We define the reduced cone as
$$M_N=\left\{ \sum_{j=1}^{N_W} \alpha_j\xi_j,\ \alpha_j\geq
0\right\},$$
which satisfies $M_N\subseteq M$ due to the assumption of $\xi_i\in M$.
In this setting, the reduced problem reads:
\begin{definition}[Reduced problem]
  Given $\mu\in{\cal
P}$, 
find $u_N^{n+1}(\mu)\in V_N$ and $\lambda_N^{n+1}(\mu)\in M_N$ 
for $0\leq n\leq L-1$ 
that satisfy 
\begin{align}\label{eq:VIRB1} 
  &\left\langle\frac{u_N^{n+1}-u_N^n}{\Delta t},v_N\right\rangle_{L^2(\Omega)} +a(\theta
  u_N^{n+1}+(1-\theta)u_N^n,v_N;\mu) - b(\lambda_N^{n+1},v_N) = f^n(v_N;\mu),\\ 
&b(\eta_N-\lambda_N^{n+1},u_N^{n+1})\geq g^n(\eta_N-\lambda_N^{n+1};\mu), \label{eq:VIRB2} 
\end{align}
for all $v_N\in V_N, \eta_N \in M_N$ and the initial value $u_N^0$ is chosen as an orthogonal projection
of $u^0$ on $V_N$, i.e., $\langle u_N^0-u^0,v_N \rangle_V=0$ for all $v_N\in V_N$.
\end{definition}
\subsection{Existence, uniqueness and reproduction property}
By the construction procedure for $V_N,W_N$ in Section~\ref{sec:implementation}, 
we will assure inf-sup 
stability of $b(\cdot,\cdot)$ on $W_N\times V_N$ and guarantee that $\beta_N\geq\beta>0$, where
\begin{eqnarray}
 \beta_N:=\inf_{\eta_N\in W_N}\sup_{v_N\in V_N}\frac{b(\eta_N,v_N)}{\|\eta_N\|_{W_N}\|v_N\|_{V_N}}, \label{eqn:reduced-inf-sup}
\end{eqnarray}
which implies the well-posedness of our 
reduced problem~(\ref{eq:VIRB1})--(\ref{eq:VIRB2}) with identical arguments 
as for the detailed saddle point problem.
Hence, we ensure existence and uniqueness of the reduced solution. For the details of the proof, we refer to~\cite{boffi,HSW12,rozza}.
A further useful property is a basic consistency argument, the 
reproduction of solutions:
\begin{lemma}[Reproduction of Solutions]
If for some $\mu$ holds $u^{n+1}(\mu) \in V_N$ and $\lambda^{n+1}(\mu)\in M_N$ 
for $0<n\leq L-1$ and $u^0(\mu)\in V_N$ then 
$$
   u_N^{n+1}(\mu) = u^{n+1}(\mu), \quad 
   \lambda_N^{n+1}(\mu) = \lambda^{n+1}(\mu), \quad \forall n=1,\ldots, L-1. 
$$
\end{lemma}
\begin{proof}
 We prove this property by induction. For $n=0$, $u^0\in V_N$ and for $u^0_N\in V_N$ we have $\langle u_N^0-u^0,v_N\rangle_V=0$, for all $v_N\in V_N$. Set $v_N=u_N^0-u^0$, then 
 \begin{equation}\nonumber
 \langle u_N^0-u^0,u_N^0-u^0\rangle_V=0,
 \end{equation}
which is true only for $u^0_N=u^0$. 
For the induction step, we assume that $u^n=u^n_N$, $\lambda^n=\lambda_N^n$. 
Then, choosing $v=v_N\in V_N\subset V$, $\eta= \eta_N\in M_N\subset M$, we 
directly obtain
\begin{align*}
  &\left\langle\frac{u^{n+1}-u_N^n}{\Delta t},v_N\right\rangle_{L^2(\Omega)} +a(\theta
  u^{n+1}+(1-\theta)u_N^n,v_N;\mu) - b(\lambda^{n+1},v_N) = f^n(v_N;\mu),\\ 
&b(\eta_N-\lambda^{n+1},u^{n+1})\geq g^n(\eta_N-\lambda^{n+1};\mu), 
\end{align*}
which implies that $(u^{n+1},\lambda^{n+1})\in V_N\times M_N$ solves the reduced problem \eqref{eq:VIRB1}--\eqref{eq:VIRB2}. Due to the uniqueness of the solution, we obtain $u_N^{n+1}=u^{n+1}$ and $\lambda_N^{n+1}=\lambda^{n+1}$.
\end{proof}

\section{A posteriori error analysis}\label{sec:error_analysis}
In this section, we present an a posteriori analysis of our RB-scheme. 
In particular, only the more challenging constraint case 
of American options is considered, as for the European option model, 
the well established standard RB-error analysis for linear parabolic 
problems \cite{grepl2005reduced,GP05,GernerVeroy2012} can be applied, 
which we omit in this work.

Again, to simplify the presentation, we omit the parameter vector $\mu$ in the notation of $a,f^n,g^n$. 
In the same way, we only consider the case $\theta=1$, that is, an implicit Euler time discretization. However the analysis presented hereafter holds for any $\theta \in (0,1]$, up to technical supplementary computations. 
We start by introducing relevant residuals and preliminary results.
\subsection{Residuals and preliminary results}
In order to evaluate the approximation errors induced by our method,
we define, for $n=0,\cdots,L-1$, the equality and inequality residuals by:
\begin{align}
 r^{n}(v)&:= \left\langle\frac{u_N^{n+1}-u_N^n}{\Delta t},v\right\rangle_{L^2(\Omega)}+a(
  u_N^{n+1},v) - b(\lambda_N^{n+1},v) - f^n(v), &&\forall v\in V\nonumber\\
 s^n(\eta)&:= b(\eta,u_N^{n+1}) - g^n(\eta), &&\forall\eta\in M. \nonumber
\end{align}
We also introduce the primal and dual errors 
\begin{equation}\label{eq:errors}
e^n_u= u_N^n-u^n, \ \ \ e^n_\lambda=\lambda_N^n-\lambda^n. 
\end{equation}
Using the linearity of $r^n$, one finds that:
\begin{equation}\label{eq:propr}
r^n(v)=\left\langle\frac{e_u^{n+1}-e_u^n}{\Delta t},v\right\rangle_{L^2(\Omega)} +a( e_u^{n+1},v) - b(e^{n+1}_\lambda,v).
\end{equation}
As a consequence of inf-sup stability, one can bound the dual error by the primal error, as stated in the next lemma.
\begin{lemma}[Primal/Dual Error Relation]\label{Lemma:PDerror}
For $n=0,\cdots,L-1$, the dual error at time step $t^n$ can be bounded by the primal error as 
\begin{equation}\label{eq:estlambda}
\|e_\lambda^{n+1}\|_{W} \leq \frac1\beta\left(\frac{C_{\Omega}}{\Delta t}\|e_u^{n+1}-e_u^n\|_{L^2(\Omega)}+\gamma_a \| e_u^{n+1}\|_V +\|r^n\|_{V'}\right).
\end{equation}
\end{lemma}
\begin{proof}
The inf-sup stability of $b(\cdot,\cdot)$ guarantees the existence of 
a $v^\star\in V,v^\star\not = 0$ such that
$$\beta \|v^\star\|_{V}\|e_\lambda^{n+1}\|_W \leq b(v^\star,e_\lambda^{n+1}).$$
Using~\eqref{eq:propr}, we find that:
\begin{eqnarray*}
 \beta \|v^\star\|_{V}\|e_\lambda^{n+1}\|_W &\leq&\left\langle\frac{e_u^{n+1}-e_u^n}{\Delta t},v^\star\right\rangle_{L^2(\Omega)} +a( e_u^{n+1},v^\star)- r^n(v^\star)\\
& \leq & \frac1{\Delta t}\|e_u^{n+1}-e_u^n\|_{L^2(\Omega)}\|v^\star\|_{L^2(\Omega)}
+\gamma_a \|e_u^{n+1}\|_V \|v^\star\|_V+\|r^n\|_{V'}\|v^\star\|_V
\end{eqnarray*}
and the result then follows from~\eqref{ineq:normequi}.
\end{proof}

\subsection{Projectors on the cone}
Let us then introduce, for $n=0,\cdots,L-1$, the Riesz-representer $\eta^n_s\in W$ of our inequality residual:
$$
\qquad \langle \eta , \eta_s^n \rangle_W = s^n(\eta),\quad  \eta \in W.
$$
Contrary to $\|r^n\|_V$, the quantity $\|s^n\|_W=\|\eta_s^n\|_W$ is not a straightforward error estimator 
component because of the inequality constraint. 
We obviously would correctly penalize 
if $s^n(\eta)>0$ for some $\eta\in M$ as desired, 
but we would also penalize $s^n(\eta)<0$ which is not necessary. 
Hence, we need to cope with the 
inherent nonlinearity induced by the inequalities. For this purpose,
we consider a family of projectors on the cone $\pi^n: W\rightarrow M$ which are assumed to satisfy, for $n=0,\cdots,L-1$, 
\begin{equation}\label{prop:pi}
\langle \pi^n(\eta^n_s),\lambda_N^{n+1}\rangle_W=0.
\end{equation} 

Having \eqref{prop:pi} and the characterization \eqref{eq:star} of the dual cone $M$ we find that
\begin{eqnarray}
\langle  e_\lambda^{n+1} , \pi^n(\eta^n_s)\rangle_W&=&\langle  \lambda_N^{n+1}-\lambda^{n+1} , \pi^n(\eta^n_s)\rangle_W= -\langle \lambda^{n+1} , \pi^n(\eta^n_s)\rangle_W\leq 0.\label{eq:cone}
\end{eqnarray}
\begin{remark}
Projectors satisfying~\eqref{prop:pi} improve
the one proposed in~\cite{HSW12}, namely one term in the error bound can be cancelled. We also refer to~\cite{weiss, WART2011} where such techniques are applied for finite element based error estimators in contact mechanics and for obstacle problems.  
\end{remark}
\subsection{A posteriori error estimators}

We are now in a position to define an a posteriori error estimator associated with our method.
\begin{theorem}\label{apostEst}
Define the a posteriori quantities
$$\delta^n_{s}=\|\eta^n_s-\pi^n(\eta^n_s)\|_W,\ \delta^n_{r}=\|r^n\|_{V'}.$$
One has:
\begin{eqnarray*}
\frac {1}{2}\|e_u^{L}\|^2_{L^2(\Omega)}  +\frac{\alpha_a}2\Delta t\sum_{n=0}^{L}\| e_u^{n}\|^2_V
&\leq&\sum_{n=0}^{L} \frac {1}{2}\left(\frac{C_\Omega\delta^n_s}{\beta}\right)^2
+\Delta t\frac{\delta^n_s\delta^n_r}\beta +\frac{\Delta t}{2\alpha_a}\left(\delta^n_r+\frac {\gamma_a\delta^n_s}{\beta}\right)^2
\\&&+\frac {1}{2}\|e_u^{0}\|^2_{L^2(\Omega)}.
\end{eqnarray*}
\end{theorem}
\begin{proof}
First, we note that we have for all $n=0,\ldots,L-1$:
\begin{equation}\label{eq:orth_N}
 s^{n}(\lambda_N^{n+1})=0,
\end{equation}
and we recall that for all $n=1,\ldots,L$:
\begin{equation}\label{eq:orth}
 b(\lambda^{n},u^{n}) = g^n(\lambda^{n}). 
\end{equation}
From~\eqref{eq:propr}, we then have:
\begin{equation}\label{eq:eq1}
\left\langle\frac{e_u^{n+1}-e_u^n}{\Delta t},e_u^{n+1}\right\rangle_{L^2(\Omega)} +a( e_u^{n+1},e_u^{n+1})=r^n(e_u^{n+1})+b(e^{n+1}_\lambda,e_u^{n+1}).
\end{equation}
Let us focus on the term $ b(e_\lambda^{n+1},e_u^{n+1} )$. Thanks to \eqref{eq:cone},~\eqref{eq:orth_N}, \eqref{eq:orth}, definition of $s^n$ and the fact that
$g^n(\lambda_N^{n+1})-b(\lambda^{n+1}_N,u^{n+1})\leq 0$ from~\eqref{eq:AOdisc2}, this term can be simplified as follows:
\begin{eqnarray}
b(e_\lambda^{n+1},e_u^{n+1} )&=&b(\lambda^{n+1}_N,u_N^{n+1})-b(\lambda^{n+1},u_N^{n+1})-b(\lambda_N^{n+1},u^{n+1})+b(\lambda^{n+1},u^{n+1})\nonumber\\
&\leq&g^n(\lambda_N^{n+1})-s^n(\lambda^{n+1})-g^n(\lambda^{n+1})-g^n(\lambda^{n+1}_N)+g^n(\lambda^{n+1})\nonumber\\
&= &-s^n(\lambda^{n+1})=s^n(e_\lambda^{n+1})=\langle e_\lambda^{n+1} , \eta^n_s\rangle_W\nonumber\\
&=&\langle e_\lambda^{n+1} , \pi^n(\eta^n_s)\rangle_W+\langle e_\lambda^{n+1}, \eta^n_s-\pi^n(\eta^n_s)\rangle_W\nonumber\\
&\leq&\| e_\lambda^{n+1}\|_W \|\eta^n_s- \pi^n(\eta^n_s)\|_W =\delta^n_{s}\| e_\lambda^{n+1}\|_W .\nonumber
\end{eqnarray}
This estimate combined with~\eqref{eq:eq1} and the coercivity of $a$, gives rise to: 
\begin{eqnarray}
\left\langle\frac{e_u^{n+1}-e_u^n}{\Delta t},e_u^{n+1}\right\rangle_{L^2(\Omega)} +\alpha_a\| e_u^{n+1}\|^2_V
\leq \delta^n_r\|e_u^{n+1}\|_V+\delta^n_s\| e_\lambda^{n+1}\|_W  .\label{eq:eq2}
\end{eqnarray}
The first term of the inequality~\eqref{eq:eq2} can be expressed as
\begin{eqnarray}
\left\langle\frac{e_u^{n+1}-e_u^n}{\Delta t},e_u^{n+1}\right\rangle_{L^2(\Omega)}&=&\frac{1}{2\Delta t}\| e_u^{n+1}\|_{L^2(\Omega)}^2-\frac{1}{2\Delta t}\| e_u^{n}\|_{L^2(\Omega)}^2\nonumber\\
&+&\frac{1}{2\Delta t}\| e_u^{n+1}-e_u^{n}\|_{L^2(\Omega)}^2  .\label{eq:eqt}
\end{eqnarray} 
Using Lemma~\ref{Lemma:PDerror} and Young's inequality, we then bound $\| e_\lambda^{n+1}\|_W$ in~\eqref{eq:eq2}:
\begin{eqnarray}
\delta^n_s\| e_\lambda^{n+1}\|_W
&\leq&\frac {1}{2\Delta t}\left(\frac{C_\Omega\delta^n_s}{\beta}\right)^2
+\frac{\delta^n_s\delta^n_r}\beta 
+\frac {1}{ 2\Delta t} \|e_u^{n+1}-e_u^n\|^2_{L^2(\Omega)} 
+\frac {\gamma_a\delta^n_s}{\beta}\|e_u^{n+1}\|_V.\nonumber
\end{eqnarray}
Combining this estimate with~\eqref{eq:eqt}, we can simplify~\eqref{eq:eq2} as follows:
\begin{eqnarray}
\frac {1}{2\Delta t}\|e_u^{n+1}\|^2_{L^2(\Omega)}  +\alpha_a\| e_u^{n+1}\|^2_V
&\leq&\frac {1}{2\Delta t} \|e_u^{n }\|^2_{L^2(\Omega)} + \frac {1}{2\Delta t}\left(\frac{C_\Omega\delta^n_s}{\beta}\right)^2
+\frac{\delta^n_s\delta^n_r}\beta \nonumber\\&+&\left(\delta^n_r+\frac {\gamma_a\delta^n_s}{\beta}\right)\|e_u^{n+1}\|_V.\label{eq:eq3}
\end{eqnarray}
Using Young's inequality gives: 
$$\left(\delta^n_r + \frac {\gamma_a\delta^n_s}{\beta}\right)\|e_u^{n+1}\|_V\leq \frac 1{2\alpha_a}\left(\delta^n_r + \frac {\gamma_a\delta^n_s}{\beta}\right)^2+\frac{\alpha_a}2\|e_u^{n+1}\|_V^2.$$
We get the result by combining the latter with~\eqref{eq:eq3}, summing the resulting inequalities from $n=0$ up to $L-1$ and multiplying the result by $\Delta t$.
\end{proof}
\begin{remark}
We point out that $\delta_s^n$ can be regarded as measure for the violation of the constraint. Thus it plays a similar role as typical penalty terms in primal methods.
\end{remark}

\section{Implementational aspects}\label{sec:implementation}
In this section, we mainly discuss the implementational aspects of solving a reduced basis problem (\ref{eq:VIRB1})--(\ref{eq:VIRB2}) associated with our detailed formulation~(\ref{eq:AOdisc1})--(\ref{eq:AOdisc2}). 
\subsection{Solution of the detailed problem}
Let us give a few remarks about the solvability of the detailed problem. The differential operator in~(\ref{eq:AOdisc1}) is of convection-diffusion type with constant (in the Black-Scholes model) or variable (in the Heston model) coefficients. Thus for the specific range of the parameters, the convective term may dominate the diffusive one, and applying the standard Galerkin method might lead to instabilities and inaccurate results. To overcome this difficulty and enhance the quality of the discrete solution, one 
can use a different discretization scheme, e.g., a streamline upwind Petrov-Galerkin method~\cite{bochev2004}. However, for our numerical simulations, we restrict the parameter range such that 
the 
diffusive term is large enough compared to the convective one and a standard Galerkin finite element method can be applied. 

For the remainder of the paper, $V$ is now a standard conforming piecewise 
linear finite element space used for the discretization 
of the variational inequality \eqref{eq:AOdisc1}--\eqref{eq:AOdisc2}. More precisely, 
consider a triangulation $\mathcal{T}_h$ of $\Omega$, consisting of $J$ simplices $K_h^j$, 
$1\leq j\leq J$, such that $\overline{\Omega}=\cup_{K_h\in\mathcal{T}_h}\overline{K}_h$. We then use a standard conforming nodal first order finite element space. 
 We define the discrete space 
$V_h:= \{v \in V| v_{|K_h^j}\in\mathbb{P}^1(K_h^j), 1\leq j\leq J \}$ of dimension $H_V:=H$.
For the sake of simplicity, we omit the index $h$ and consider $V$ to be a discrete finite element space. 
We associate the basis functions $\phi_i \in V$ with its Lagrange 
node $m_i\in\overline{\Omega}$, i.e., $\phi_i(m_j) = \delta_{ij}, i,j=1,\ldots,H$.

For the discretization of the Lagrange multipliers in $M\subset W$, we use 
dual basis functions~\cite{Woh00a}, that is, we consider a dual finite
element basis $\chi_j$ of $W:=V'$, so that
$b(\phi_i,\chi_j)=\delta_{ij}$, $i,j=1,\ldots,H_W=H$. The cone $M$ is
defined by:
$M=\left\{\sum_{i=1}^{H_W}\eta_i\chi_i,\ \eta_i\geq 0 \right\}.$

As these spaces are assumed to be sufficiently accurate a priori, that is, the 
finite element discretization error is neglegible compared to the 
reduction error, we do not discriminate notationally between the true and 
the finite element spaces. However a more extensive analysis has to take into account also the quality of the detailed solution.

\subsection{Solution algorithm for the reduced problem}\label{sec:SolAlg}
We shall now present a method to solve the problem~(\ref{eq:VIRB1})--(\ref{eq:VIRB2}). 
The approach we follow is mainly based on the Primal-Dual-Active-Set-Strategy~\cite{wohlmuth,kunisch}, which is equivalent to semi-smooth Newton method, thus a superlinear convergence of the algorithm can be achieved. 

We expand the solution $\left(u_N^n(\mu),\lambda_N^n(\mu)\right)$ of~(\ref{eq:VIRB1})--(\ref{eq:VIRB2}) as $u_N^n(\mu)=\sum_{j=1}^{N_V}\overline{u}_{N,j}^n\psi_j$ and $\lambda_N^n(\mu)=\sum_{j\prime=1}^{N_W}\overline{\lambda}_{N,j\prime}^n\xi_{j\prime}$ with the coefficient vectors $U_N^{n}=\left(\overline{u}_{N,j}\right)_{j=1}^{N_V}\in \R^{N_V}$, $\Lambda_N^{n}=\left(\overline{\lambda}_{N,j'}\right)_{j'=1}^{N_W} \in \R^{N_W}$. We also introduce the following set of notations: For $\nu\in[-1,1]$, denote
by $\mathcal{M}$, ${\mathcal{A}}^\nu(\mu)$ and
${\cal B}$ the matrices of coefficients $\left(\mathcal{M}\right)_{i,j}= \langle
\psi_i,\psi_j\rangle_V$, $\left({\mathcal{A}}^\nu(\mu)\right)_{i,j}=\langle
\psi_i,\psi_j\rangle +\nu \Delta t a(\psi_i,\psi_j ; \mu)$ and
$\left({\mathcal{B}}\right)_{i,j'}=b(\xi_{j'},\psi_i)$  with $1\leq i,j\leq
N_V$ and $1\leq j'\leq N_W$, respectively. Denote also by ${\cal F}^n(\mu)$ and
${\cal G}^n(\mu)$ the vectors of components $f^n(\psi_i;\mu)$ and
$g^n(\xi_{j'};\mu)$, respectively. 
Let $U$ and $\Lambda$ be generic
coefficient vectors of length $N_V$ and $N_W$, respectively, then we introduce the function
$\varphi(U,\Lambda)= \Lambda-\max (0,\Lambda-c({\cal B}^TU - {\cal
  G}^n(\mu)))$, where $c>0$ and $\max(\cdot)$ applies componentwise. With the use of these notations, the problem~(\ref{eq:VIRB1})--(\ref{eq:VIRB2}) can be rewritten in the algebraic form
\begin{align}
{\cal A}^\theta(\mu) U_N^{n+1}-{\cal B}\Lambda_N^{n+1}&={\cal
  A}^{\theta-1}(\mu) U_N^{n}+{\cal F}^n(\mu),& \label{eq:VIRB7}\\ 
  \varphi(U_N^{n+1},\Lambda_N^{n+1})&=0,&\label{eq:VIRB8}
\end{align}
and the initial data is obtained by solving ${\cal M} U^0_N=\left( \langle
u^0,\psi_j\rangle_V\right)_{j=1}^{N_V}$. To employ the the Primal-Dual-Active-Set-Strategy, we introduce the active and inactive sets
\begin{align*}
{A}(U,\Lambda)&=\left\{p:\ 1\leq p \leq N_W,\ \left(\Lambda-c({\cal B}^TU - {\cal
  G}^n(\mu)\right)_p\geq 0 \right\},\\
{I}(U,\Lambda)&=\left\{p:\ 1\leq p \leq N_W,\ \left(\Lambda-c({\cal B}^TU - {\cal
  G}^n(\mu)\right)_p < 0  \right\},
\end{align*}
where we have denoted by $(\cdot)_p$ the $p$-th component of a vector. 
A Newton iteration can then be applied, which gives rise to the next algorithm.
\begin{algorithm}[Time solver for the reduced system]
Given a tolerance $\varepsilon>0$ and initial conditions $(U_N^{0},\Lambda_N^0)$, the trajectory
$(U_N^{n},\Lambda_N^n)$, $n=0,\cdots,L$ is computed recursively as
follows. Suppose that at time step $t_n=n\Delta t$, $(U_N^{n},\Lambda_N^{n})$ is known.
\begin{enumerate}
\item Set $(U_N^{n+1,0},\Lambda_N^{n+1,0}):=(U_N^{n},\Lambda_N^{n})$ and $Tol=+\infty$. 
\item While $Tol>\varepsilon$, do
\begin{enumerate}
\item Define $(U_N^{n+1,k+1},\Lambda_N^{n+1,k+1})$ as the solution of:
\begin{align*}
{\cal A}^\theta(\mu)U_N^{n+1,k+1}-{\cal B}\Lambda_N^{n+1,k+1}&={\cal
  A}^{\theta-1}(\mu) U_N^{n}+{\cal F}^n(\mu),  \\ 
\left( {\cal B}^T U_N^{n+1,k+1}\right)_p&=\left({\cal
  G}^n(\mu)\right)_p, \ \ \ \ \forall p\in  {A}(U_N^{n+1,k},\Lambda_N^{n+1,k}),\\
\left( \Lambda_N^{n+1,k+1}\right)_p&=0,  \ \ \ \ \ \ \ \  \ \ \ \ \ \ \forall p\in  {I}(U_N^{n+1,k},\Lambda_N^{n+1,k}).
\end{align*}
\item Set $Tol=\|U_N^{n+1,k+1}-U_N^{n+1,k}\|_V+\|\Lambda_N^{n+1,k+1}-\Lambda_N^{n+1,k} \|_W$.
\end{enumerate}
\item Define $(U_N^{n+1},\Lambda_N^{n+1})=(U_N^{n+1,k+1},\Lambda_N^{n+1,k+1})$.
\end{enumerate}
\end{algorithm}
Note that $U_N^{n+1,k}$ and $\Lambda_N^{n+1,k}$ do not appear in the
iteration (except in the definitions of the sets  $
{A}(U_N^{n+1,k},\Lambda_N^{n+1,k})$ and $
{I}(U_N^{n+1,k},\Lambda_N^{n+1,k})$) as the functions involved
in~(\ref{eq:VIRB7})--(\ref{eq:VIRB8}) are either linear or piecewise linear. Using the previous time-step solution as a start iterate the active set is already a very good guess and thus we can expect a good convergence of the Newton iteration.
\subsection{Projectors}\label{sec:RBproj} 
For the a posteriori error estimators, we require suitable projectors on the cone.
In the case of our finite element basis, we simply choose them by their discrete vectorial 
representation in the dual finite element space as
$$ \pi^n(\eta):=I_{|\{\Lambda_N^{n+1}=0\}} \circ \left([{\cal B}^T U - {\cal G}^n(\mu)]_+\right),$$
where for $p=1,\ldots,N_w$, $I_{|\{\Lambda_N^{n+1}=0\}}$ is defined by 
$\left(I_{|\{\Lambda_N^{n+1}=0\}}\right)_p=0$ if $\left(\Lambda_N^{n+1}\right)_p\neq 0$, and  
$\left(I_{|\{\Lambda_N^{n+1}=0\}}\right)_p=1$ else. Given a coefficient vector $X$, we denote by $[X]_+$ its positive part, that is $\left([X]_+\right)_p=0$, if  $\left(X\right)_p\leq 0$ and $\left([X]_+\right)_p=\left(X\right)_p$ else.
With this definition, the property (\ref{prop:pi}) is trivially fulfilled. 
\subsection{Computation of primal and dual reduced basis}\label{sec:RBcomp} 
In this section, we present a method to build the reduced primal basis $\Psi_N\subset V$ and 
dual basis $\Xi_N \subset M$. 
The approach we follow consists in building iteratively and simultaneously 
the reduced primal and  dual basis in a greedy fashion, as presented in Algorithm 2. 
We consider a finite training set  ${\cal P}_{train}\subset {\cal P}$  small enough such that it can be 
scanned quickly, but sufficiently large to represent the parameter space well.
After a (quite arbitrary) choice of the initial reduced spaces, we proceed in a greedy loop:
For a given primal and dual space at stage $k$, we identify the parameter 
vector $\mu_{k+1}$ in the training set 
that currently leads to the worst reduced basis approximation. 
For this parameter, new primal and dual basis vectors are generated, and this loop 
is repeated $N_{\max}\in \N$ times.
To be more precise, the 
selection of new primal vectors is based on the idea of the POD-Greedy procedure, which meanwhile is 
standard for RB-methods for time-dependent parametric problems \cite{haasdonk08,HSW12b} and has provable 
convergence rates \cite{Ha13}. The main step consists in computing the complete simulation of 
the worst-resolved trajectory, extracting the new information by an orthogonal projection on the current 
primal space and a compression by a Proper Orthogonal Decomposition (POD) to a single vector, 
the so called dominant POD-mode.
In the algorithm $\Pi_{ V^k_N}$ denotes the orthogonal
projection on $ V^k_N$ with respect to $\langle \cdot,\cdot\rangle_V$ and the dominant POD-mode is given as
\begin{align*}
POD_1\left( \{v^n\}_{n=0}^L\right) :=\arg\min_{\|z\|_V=1}\sum_{n=0}^L\| v^n - \left\langle v^n, z\right\rangle_V z \|_V^2.
\end{align*}
The new dual basis vectors are selected using an Angle-Greedy argument, which aims at maximizing the 
volume of the resulting cone \cite{HSW12b}. This means, we include the snapshot showing the largest  
deviation from the current reduced dual space,
that is, the vector $\lambda^{n_{k+1}}(\mu_{k+1})$ that 
maximizes $\measuredangle \left( \lambda^{n_{k+1}}(\mu_{k+1}) , W_N^k\right)$ where $\measuredangle (\eta,Y)
:= \arccos  ||\Pi_{Y}\eta||_W / ||\eta||_W$ denotes the angle
between a vector $\eta\in W$ and a linear space $Y\subset W$. 

We point out that our system \eqref{eq:VIRB1}--\eqref{eq:VIRB2} has a saddle point structure. Thus 
ignoring the dual basis in the construction of the primal basis may lead 
to a reduced system for which the stability cannot be guaranteed.
To guarantee the inf-sup stability of our approach, we follow the ``inclusion of supremizers'' 
idea introduced in~\cite{rozza} for the Stokes problem.
This enrichment consists in including $B\xi_{k+1}$ into the primal space, 
where $B\xi_{k+1} \in V$ is the solution of $b(\xi_{k+1},v) = \langle B\xi_{k+1},v\rangle_V$, for $v\in V$. 
Then $v=B\xi_{k+1}$ is the element that supremizes the expression
$\langle B\xi_{k+1},v\rangle_V$, hence it is called a ``supremizer''.
This extension ensures that the reduced inf-sup condition (\ref{eqn:reduced-inf-sup})
is satisfied, which can be proven following the lines of 
\cite{HSW12}, and hence the reduced problem is well-posed. 
We conclude by defining the final reduced space $V_N := {\rm span} \Psi_N$ of 
dimension $N_V :={\rm dim} V_N$.

The selection of the ``worst'' parameter in the greedy loop requires a measure $E(\mu)$, which is an 
error estimator that can be chosen as one of these three quantities:
\begin{align}
E_{L^2}^{true}(\mu)&={\Delta t\sum_{n=0}^{L}\|e_u^{n}\|^2_{V}}  \label{Est:L2},\\
E_{energy}^{true}(\mu)&={\frac {1}{2}\|e_u^{L}\|^2_{L^2(\Omega)}  +\frac{\alpha_a(\mu)}2\Delta t\sum_{n=0}^{L}\| e_u^{n}\|^2_V}\label{Est:true},\\
E_{energy}^{Apost}(\mu)&={\sum_{n=0}^{L} \frac {1}{2}\left(\frac{C_\Omega\delta^n_s}{\beta}\right)^2
+\Delta t\frac{\delta^n_s\delta^n_r}\beta +\frac{\Delta t}{2\alpha_a(\mu)}\left(\delta^n_r+\frac {\gamma_a(\mu)\delta^n_s}{\beta}\right)^2
+\frac {1}{2}\|e_u^{0}\|^2_{L^2(\Omega)}}.\nonumber\\\label{Est:Apost}
\end{align}
Hence $E_{L^2}^{true}$ is the squared true $L^2$-error, $E_{energy}^{true}$ is the true 
error measured in a space-time energy, and $E_{energy}^{Apost}$ is the a posteriori error bound 
from the previous section, cf.\ Theorem~\ref{apostEst}. 
With these notations, the resulting POD-Angle-Greedy algorithm 
is fully specified.
\begin{algorithm}[POD-Angle-Greedy algorithm]\label{alg:global-greedy}
Given $N_{\max}>0$, ${\cal P}_{train}\subset {\cal P}$
\begin{enumerate}
\item choose arbitrarily $\mu_1\in {\cal P}_{train}$ and $n_1 \in 1,\ldots,L$
\item set $\xi_1:= \lambda^{n_1}(\mu_1)/\norm{\lambda^ {n_1}(\mu_1)}_W$, $\Xi_N^1=\left\{ \xi_1 \right\}$, $W_N^1:={\rm span}(\Xi_N^1)$,
\item set $\Psi_N^1:={\rm orthonormalize}\left\{u^{n_1}(\mu_1),B \xi_1 \right\}$,
  $ V^1_N:={\rm
  span}(\Psi_N^1)$,
\item for $k=1,\ldots, N_{\max}-1$, do
\begin{enumerate}
\item define $ \mu_{k+1}:={\rm argmax}_{\mu\in{\cal P}_{train}}{\left( E(\mu) \right)}$,
\item find $n_{k+1}:={\rm argmax}_{n=1,\ldots,L} \left( \measuredangle \left( \lambda^n(\mu_{k+1}) , W_N^k\right)\right),$
\item set $\xi_{k+1}:= \lambda^{n_{k+1}}(\mu_{{k+1}}) / \norm{\lambda^{n_{k+1}}(\mu_{{k+1}})}_W$, \\ 
        $\Xi^{k+1}_N:=\Xi^k_N\cup \{\xi_{k+1}\}$, $W_N^{k+1}:={\rm span}(\Xi_N^{k+1})$,
\item define $\tilde \psi_{k+1}:=POD_1\left(\left\{u^n(\mu_{k+1})-\Pi_{ V^k_N}(u^n(\mu_{k+1}))\right\}_{n=0,\ldots,L}\right)$, \\
set $\Psi_N^{k+1}:={\rm orthonormalize}\left( \Psi_N^ {k}  
 \cup \left\{ \tilde \psi_{k+1}, B \xi_{k+1} \right\} \right)$,
  $ V^{k+1}_N:={\rm span}(\Psi_N^{k+1})$,
\end{enumerate}
\item define $\Xi_N:=\Xi_N^{N_{\max}}$, $W_N:={\rm span}(\Xi_N)$, $N_W:={\rm dim}(W_N)$,
\item define $\Psi_N:=\Psi_N^{N_{\max}}$, 
$ V_N:={\rm span} (\Psi_N)$, $N_V:={\rm dim} (V_N)$.
\end{enumerate}
\end{algorithm}

Let us give some additional comments on this procedure.
First note, that the extension steps in the algorithm are only performed if the sets remain 
linearly independent. Another way of stating this is to say that if a supremizer is already contained in the current space, it 
will not be inserted again in order to maintain the linear independence of the reduced bases. 
This results in $|\Psi_N^{k}|\leq 2k$ and $|\Xi_N^k| \leq k$.

Note also that the orthonormalization steps can either be done by a simple Gram-Schmidt 
orthonormalization, or a singular value decomposition (SVD). In particular the former is very attractive, as 
$V_N^k$ and $\tilde \psi_{k+1} $ are already orthonormal, hence in Step 4.(d) it is sufficient to 
orthonormalize the single vector $B\xi_{k+1}$.

Further, we emphasize that the supremizer enrichment in practice also may 
be skipped: This results in 
smaller and hence faster reduced models, by accepting the loss of a theoretical stability guarantee, see, e.g.,~\cite{HSW12}. 

Using the error estimator $E_{energy}^{true}(\mu)$ or $E_{L^2}^{true}(\mu)$ corresponds to work 
with quantities associated with the true error and is consequently more 
expensive, as the full trajectories for all parameters in the training set must be precomputed. This is in contrast to the case of working 
with $E_{energy}^{Apost}(\mu)$, which only involves the reduced solver 
used in the online phase, and only detailed trajectories for the 
selected parameters must be computed. Hence, when using $E_{energy}^{Apost}(\mu)$, the training set ${\cal P}_{train}$ can be potentially 
chosen much larger and hence much more representative for the 
complete parameter domain.

We finally give a short comment on the European option case: In that case, we use the $E_{L^2}^{true}(\mu)$ measure for 
selecting the worst parameter. Due to the absence of the Lagrange multiplier, we do not obtain a saddle point structure of the problem. Hence, we omit the supremizer enrichments and the Angle-Greedy step, and the algorithm then reduces to the classical (strong) POD-Greedy algorithm. 


\section{Numerical results}\label{sec:results}

\subsection{Numerical setting}
We start with a description of the numerical
values and methods we consider for the offline solver for Black-Scholes and Heston models. 

We use the $\theta$-scheme presented in
Section~\ref{sec:var_form} for the
time-discretization with $\theta=\frac{1}{2}$ for European and $\theta=1$ for American options. 
The time domain $[0,T]=[0,1]$ is discretized with a uniform mesh of step size $\Delta t:= T/L$, $L=20$.
The space domain is set to $\Omega=(S_{\min},S_{\max})=(0,300)\subset\mathbb{R}^1$ and $\Omega=(v_{\min},v_{\max})\times(x_{\min},x_{\max})=(0.0025,0.5)\times(-5,5)\subset\mathbb{R}^2$ for the Black-Scholes and Heston model, respectively. The Black-Scholes model is treated in the original $S$ variable, while the Heston model is considered with respect to the log-transformation $x=\log\left(\frac{S}{K}\right)$.
As a consequence, we use the $H^1$ norm for the Heston model and the weighted $H^1$ norm
$u\mapsto \left(\int_\Omega u^2 + S^2 (\partial_s u)^2  dS\right)^{1/2}$
for the Black-Scholes model.

 For the Black-Scholes model we set $H=200$ nodes and the Heston model $H=49\times 97=4753$ nodes.  
To build the basis, we consider a subset ${\cal P}_{train}$ of $\cal P$, which is specified for the Black-Scholes model as in \eqref{eq:par_domBS} and for the Heston model as in \eqref{eq:par_domH} for European options and as in~\eqref{eq:par_domHA} for American options.
\begin{align}
{\cal  P}&\equiv[0.0475,0.0525]\times[0.0014,0.0016]\times[0.4750,0.5250]
\label{eq:par_domBS}&\subset\mathbb{R}^3,\\
\mathcal{P}&\equiv[0.1,0.4]\times[0.21,0.9]\times[0.08,0.15]\times[1.2, 3]\times[0.01,0.2]&\subset\mathbb{R}^5,\label{eq:par_domH}\\
\mathcal{P}&\equiv[0.6,0.9]\times[0.21,0.9]\times[0.16,0.25]\times[3, 5]\times[0.01,0.2]&\subset\mathbb{R}^5\label{eq:par_domHA}.
\end{align}	
For the Black-Scholes setting we use $\mathcal{P}_{train}$ composed of $4^3$ values chosen equidistantly distributed in $\mathcal{P}$, while for the Heston model the set $\mathcal{P}_{train}$ is varying and specified in every particular case. We define our model parameters $\mu\in\mathcal{P}$ as $\mu\equiv(\rho,q,\sigma)$ and $\mu\equiv(\xi,\rho,\gamma,\kappa,r)$ for the Black-Scholes and Heston models, respectively. The strike $K$ is not considered as a model parameter, since it scales the value of the option. We set $K=100$ and $K=1$ for the Black-Scholes and Heston model, respectively.
In order to design the reduced primal and dual bases, we use the POD-Angle-Greedy algorithm with different error measures $E(\mu)$.

\subsection{European options}
We first consider the numerical results and a performance of the reduced basis method for the simpler linear case of the European call option with the Heston model. 
To illustrate the motivation to apply the reduced basis approach, we demonstrate the variability of the solution in parameter and time. In Figure~\ref{Fig:snapshots_heston} (top left) the detailed solution for a fixed parameter value at the final time $t=T$ is presented. The evolution of the solution for different time $t$ is shown in Figure~\ref{Fig:snapshots_heston} (top right) and the variation with respect to  parameters in Figure~\ref{Fig:snapshots_heston} (bottom). 
\begin{figure}
\includegraphics[width=.5\linewidth]{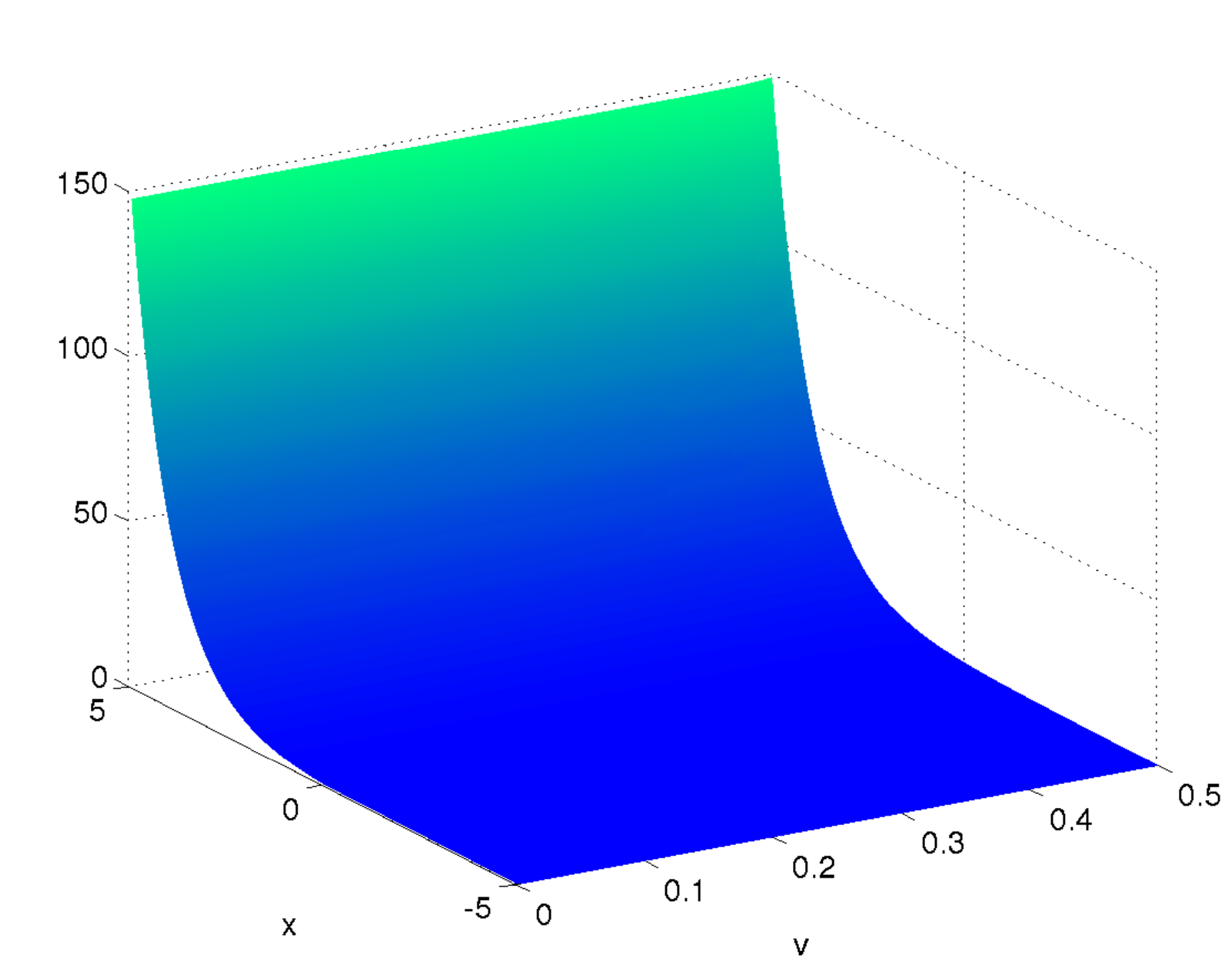}
 \includegraphics[width=.55\linewidth]{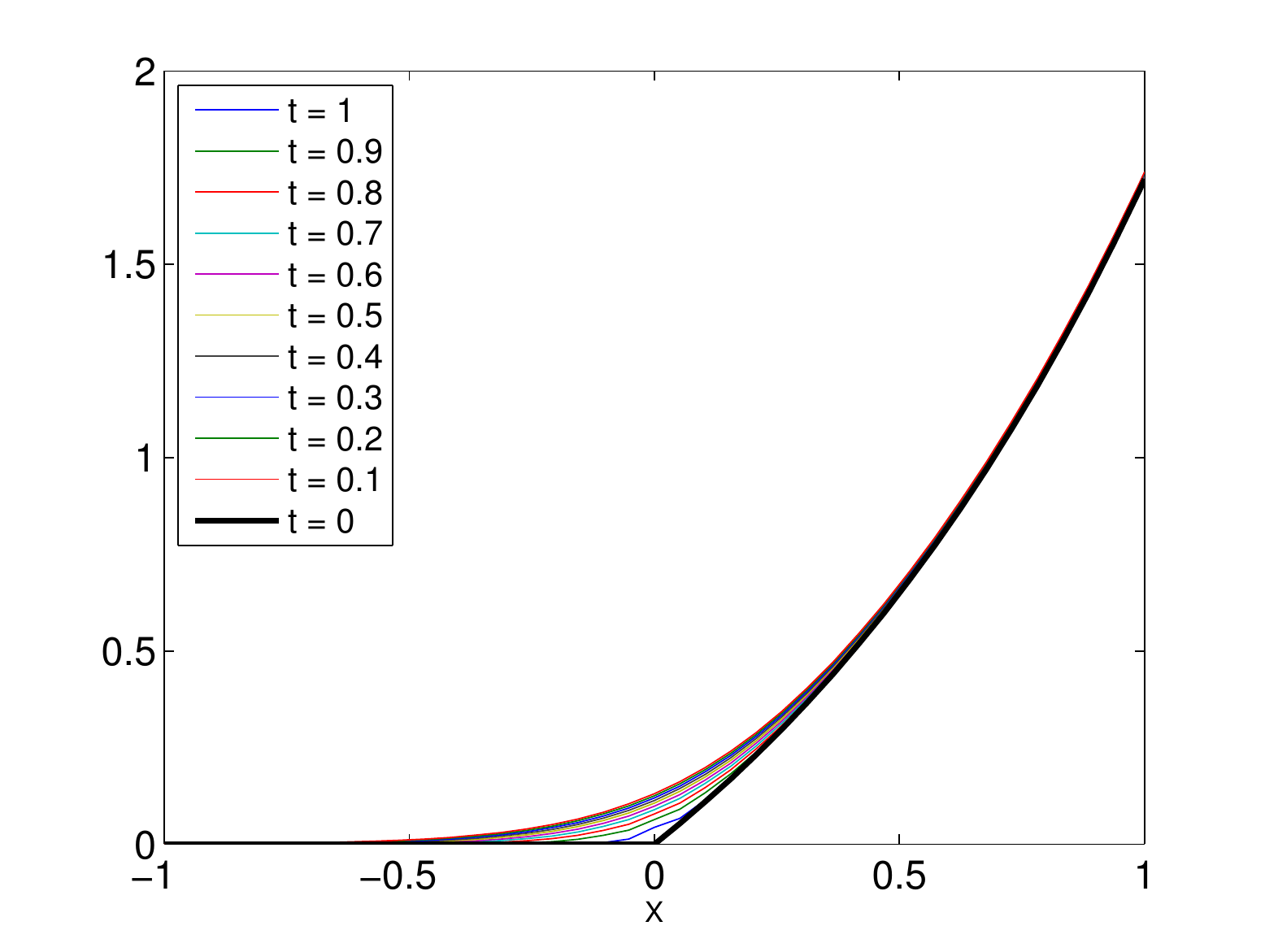}\\
 \includegraphics[width=.5\linewidth]{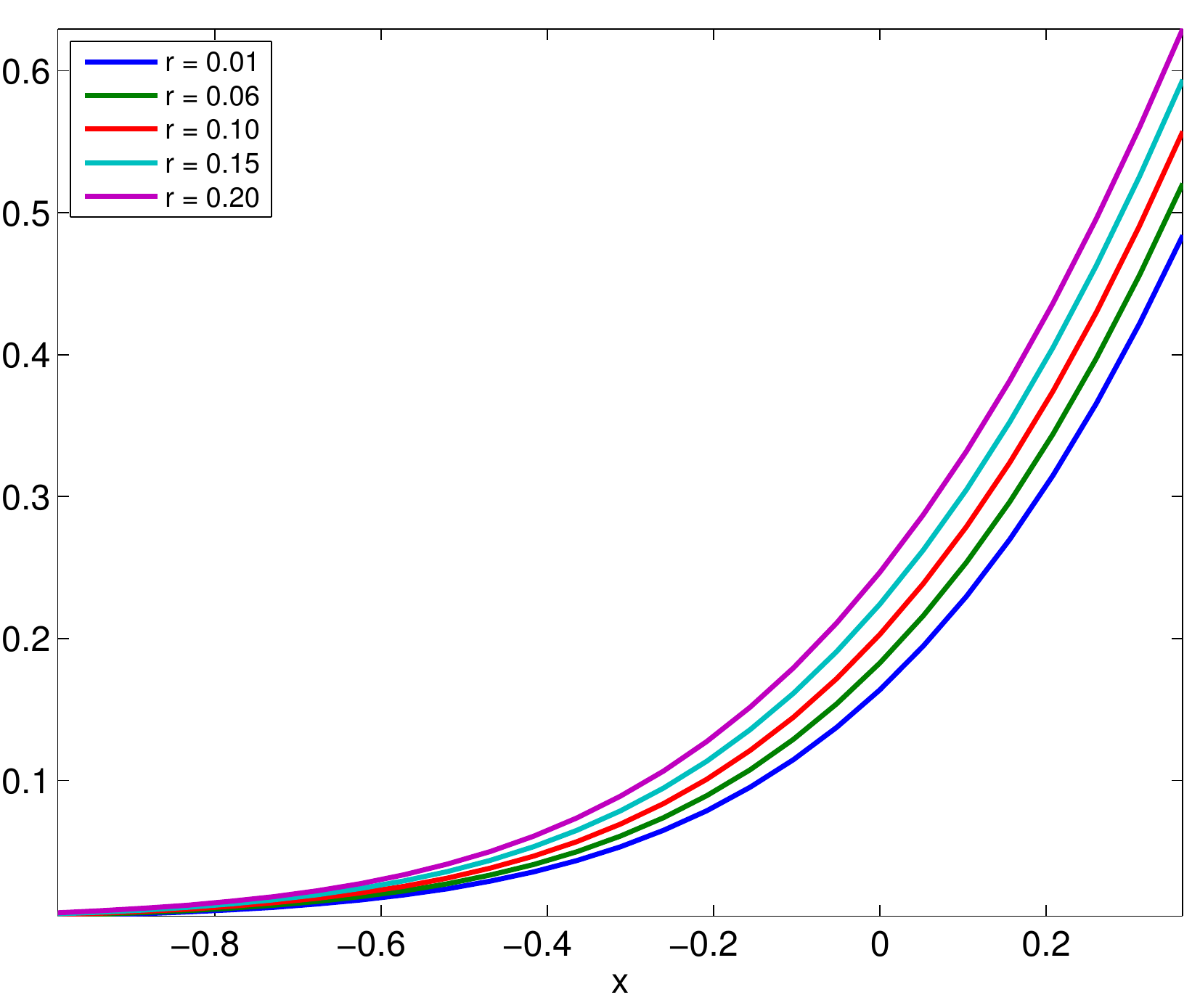}
\includegraphics[width=.5\linewidth]{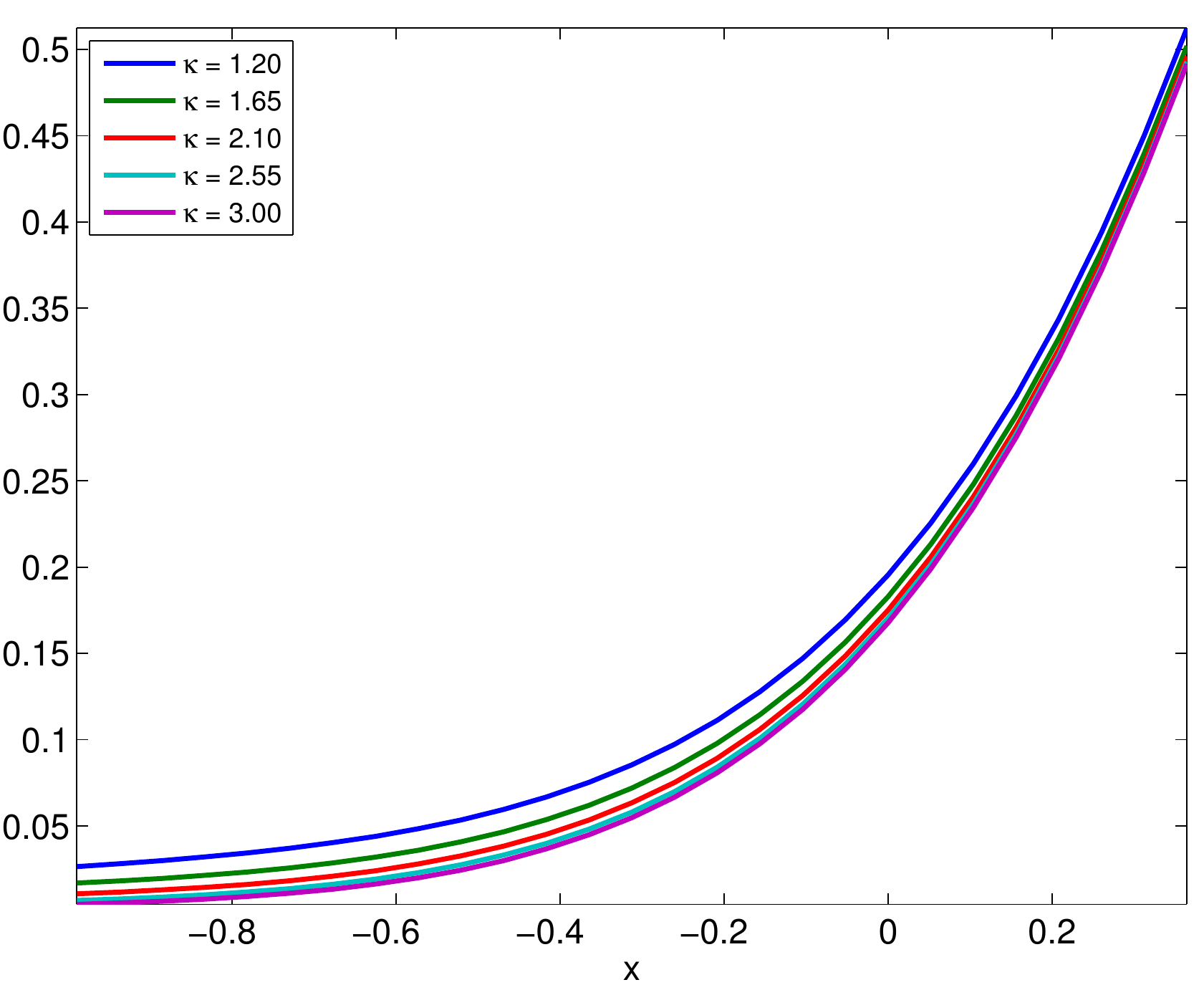} 
\caption{Top: The value of the European call option with the Heston model for $\mu=(0.4, 0.55, 0.06, 2.5, 0.0198)^T$ at $t=T=1$ (left) and a time evolution of the option at $v=0.1683$ (right). Bottom: Snapshots of the solution for different parameter values  extracted at a fixed volatility $v=0.1683$.}
\label{Fig:snapshots_heston}
\end{figure}
After computing the set of snapshots which consists of the detailed solutions for different $\mu_i\in\mathcal{P}_{train}$, we employ the POD-Angle-Greedy algorithm with $E_{L^2}^{true}(\mu)$ as a selection criterion to construct the reduced basis space $V_N$. As mentioned earlier, this procedure corresponds to the standard strong POD-Greedy algorithm.   

We test the reduced basis approach for different dimension of the parameter domain $\mathcal{P}\subset\mathbb{R}^d$, $d=2,3,5$, namely we consider $\mu=(\gamma,\kappa)\in\mathbb{R}^2$, $\mu=(\gamma,\kappa,r)\in\mathbb{R}^3$ and $\mu=(\xi,\rho,\gamma,\kappa,r)\in\mathbb{R}^5$. For each choice of $\mu$, the remaining parameter values are assumed to be fixed and taking the value from the default parameter vector $\mu^*=(0.3, 0.21, 0.095, 2, 0.0198 )^T$. 
In our first test, we consider $\mu=(\kappa, \gamma)$ and $|\mathcal{P}_{train}|=15^2=225$ equidistantly distributed points. The first six orthonormal reduced basis vectors produced by the algorithm are presented in Figure~\ref{Fig:RB_2D_basis}. 
\begin{figure}
\includegraphics[width=0.75\linewidth,height=7cm]{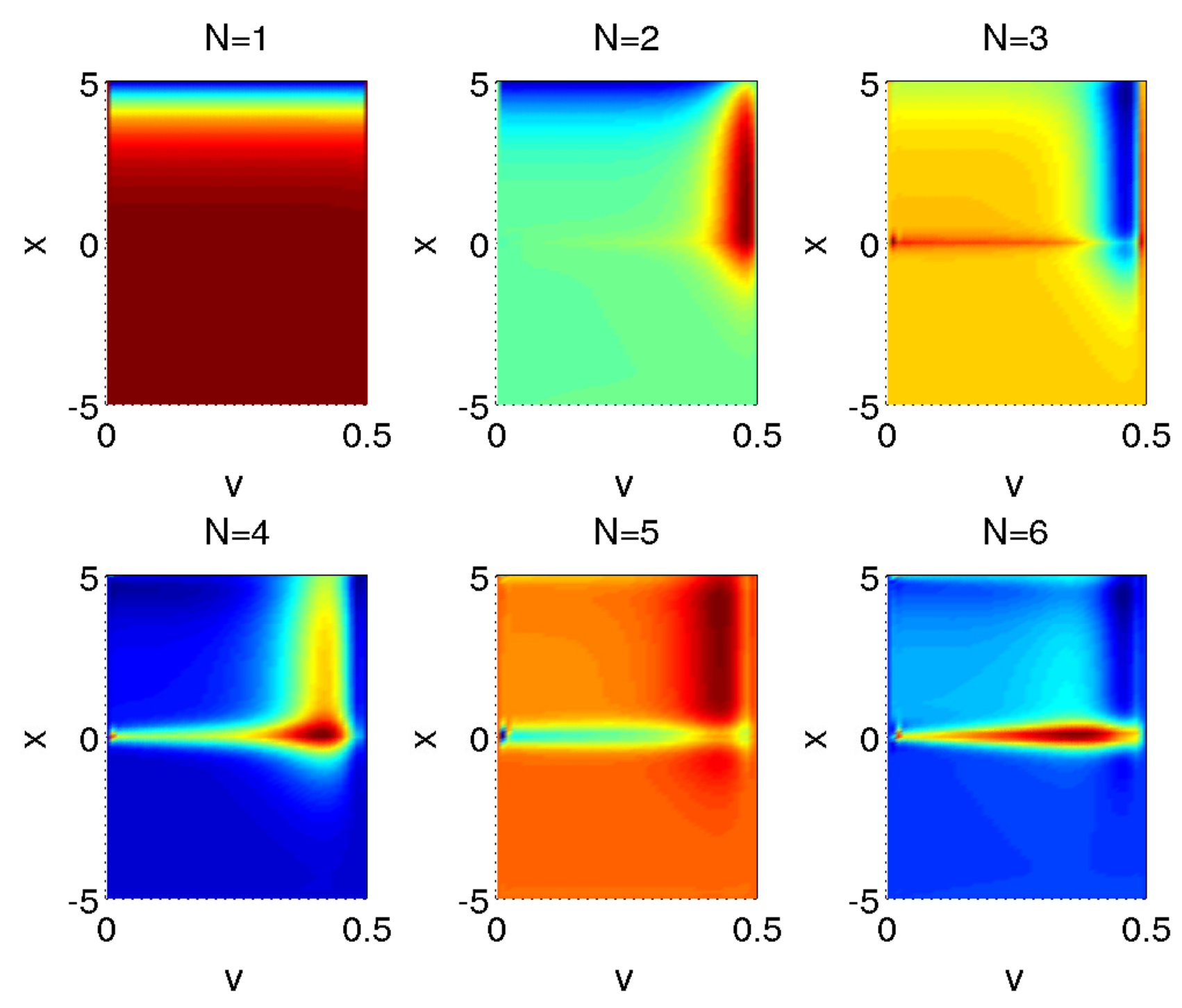} \centering
\caption{First six vectors of the reduced basis $\left\{\psi_k\right\}_{k=1}^{N_V}\subset\Psi_N$, obtained by Algorithm~\ref{alg:global-greedy} for the European option with the Heston model using $E(\mu)=E^{true}_{L^2}(\mu)$ for the case of $\mu=(\gamma,\kappa)$. }
\label{Fig:RB_2D_basis}
\end{figure}

To quantify the efficiency of the reduced basis method, we investigate the error decay when increasing the dimension $N_V$. For each reduced model, we compute the maximal error $\max_{\mu\in\mathcal{P}_{test}}\left\{E^{true}_{L^2}(\mu)\right\}$  over a random test set $|\mathcal{P}_{test}|=400$. In Figure~\ref{Fig:RB_2D_error}, we observe that the error plotted versus the dimension of the reduced model decays exponentially.
\begin{figure}
 \includegraphics[width=0.5\linewidth]{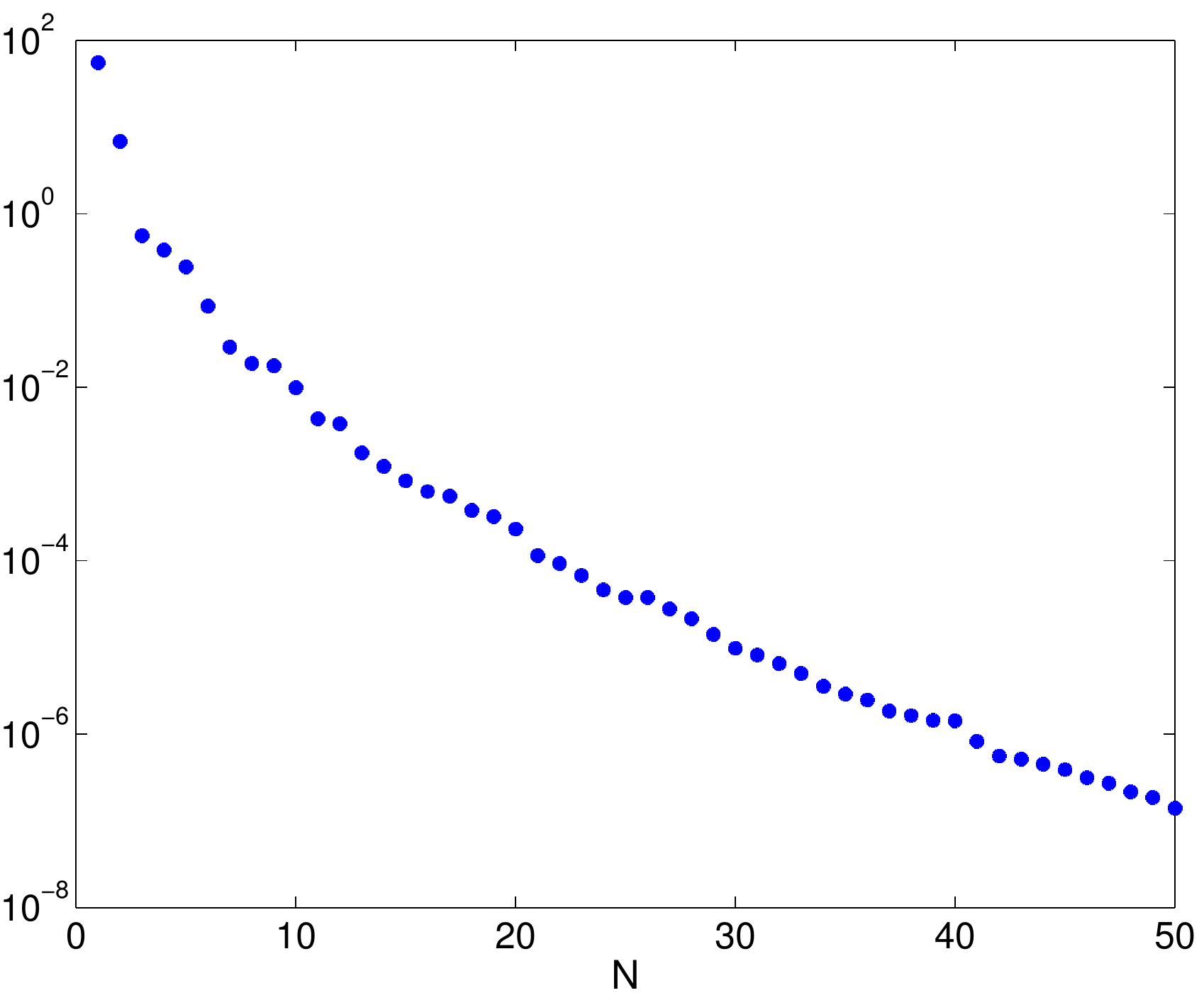}
  \includegraphics[width=0.5\linewidth]{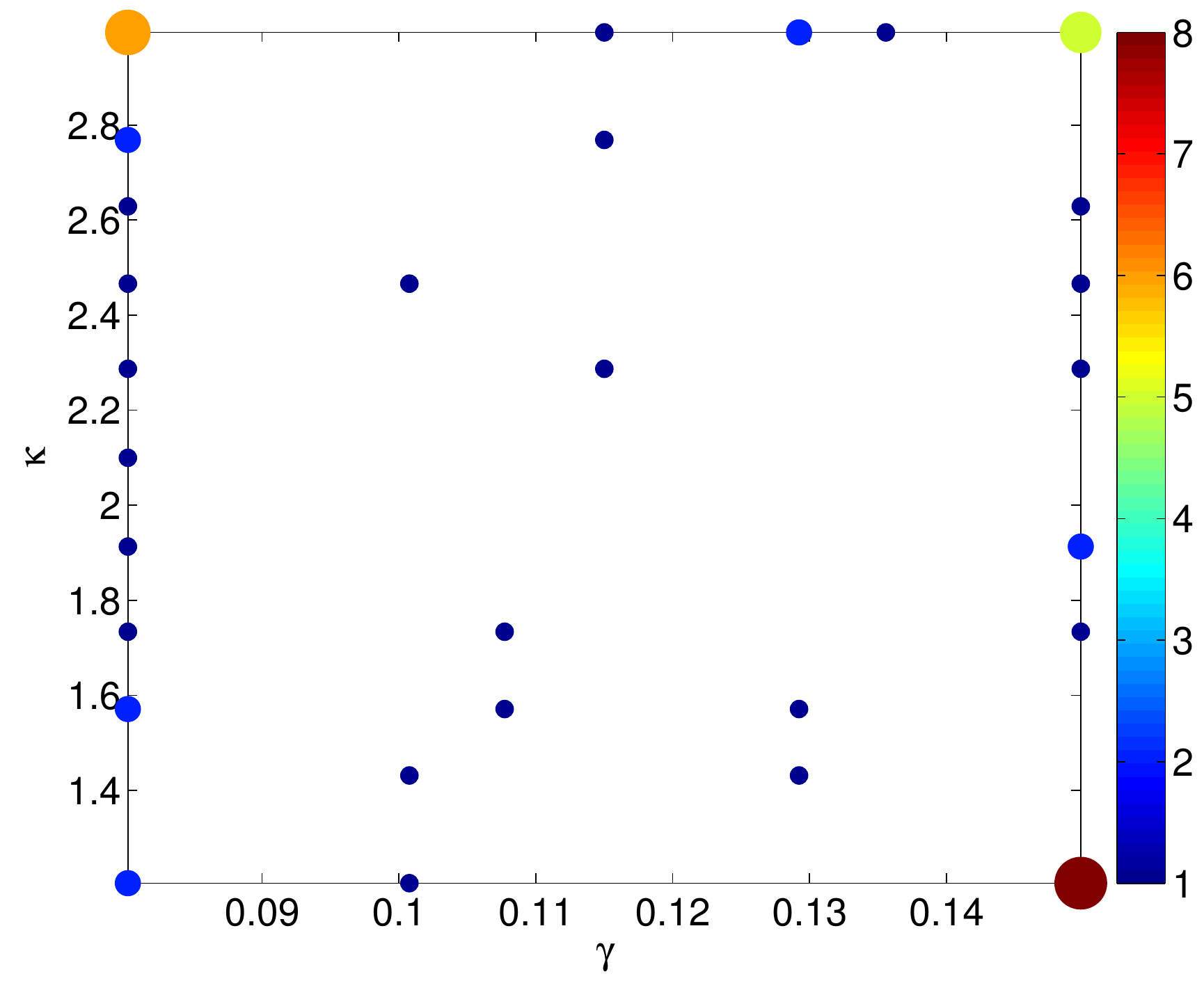}
 \caption{Left: Evolution of $\max_{\mu\in\mathcal{P}_{test}}\left\{E^{true}_{L^2}(\mu)\right\}$  for the European option with the Heston model and $\mu=(\gamma,\kappa)$. The test grid is random $|\mathcal{P}_{test}|=400$. Right: Plot of the selected parameters $\mu_1,...,\mu_N\in\mathcal{P}_{train}$ and their frequency of the selection in the construction of the reduced basis in Algorithm~\ref{alg:global-greedy}. The train set composed of $|\mathcal{P}_{train}|=15^2=225$ equidistantly distributed points.}
\label{Fig:RB_2D_error}
\end{figure}
The analogous results for other choices of $\mu=(\gamma,\kappa, r)$ and $\mu=(\xi,\rho,\gamma,\kappa,r)$ are presented in Figure~\ref{Fig:RB_3D_error} and Figure~\ref{Fig:RB_5D_error}. For both cases the exponentially decaying behavior of the error is shown, however it can be observed that the convergence is slower for larger dimension of the parameter, which is explained by the increasingly complex parameter dependence of the model. In Figure~\ref{Fig:RB_5D_error}, we also present the evolution of the train error $\max_{\mu\in\mathcal{P}_{train}}\left\{E^{true}_{L^2}(\mu)\right\}$ used for the reduced basis construction in the step $4.(a)$ of Algorithm~\ref{alg:global-greedy}. 

\begin{figure}
  \includegraphics[width=0.5\linewidth]{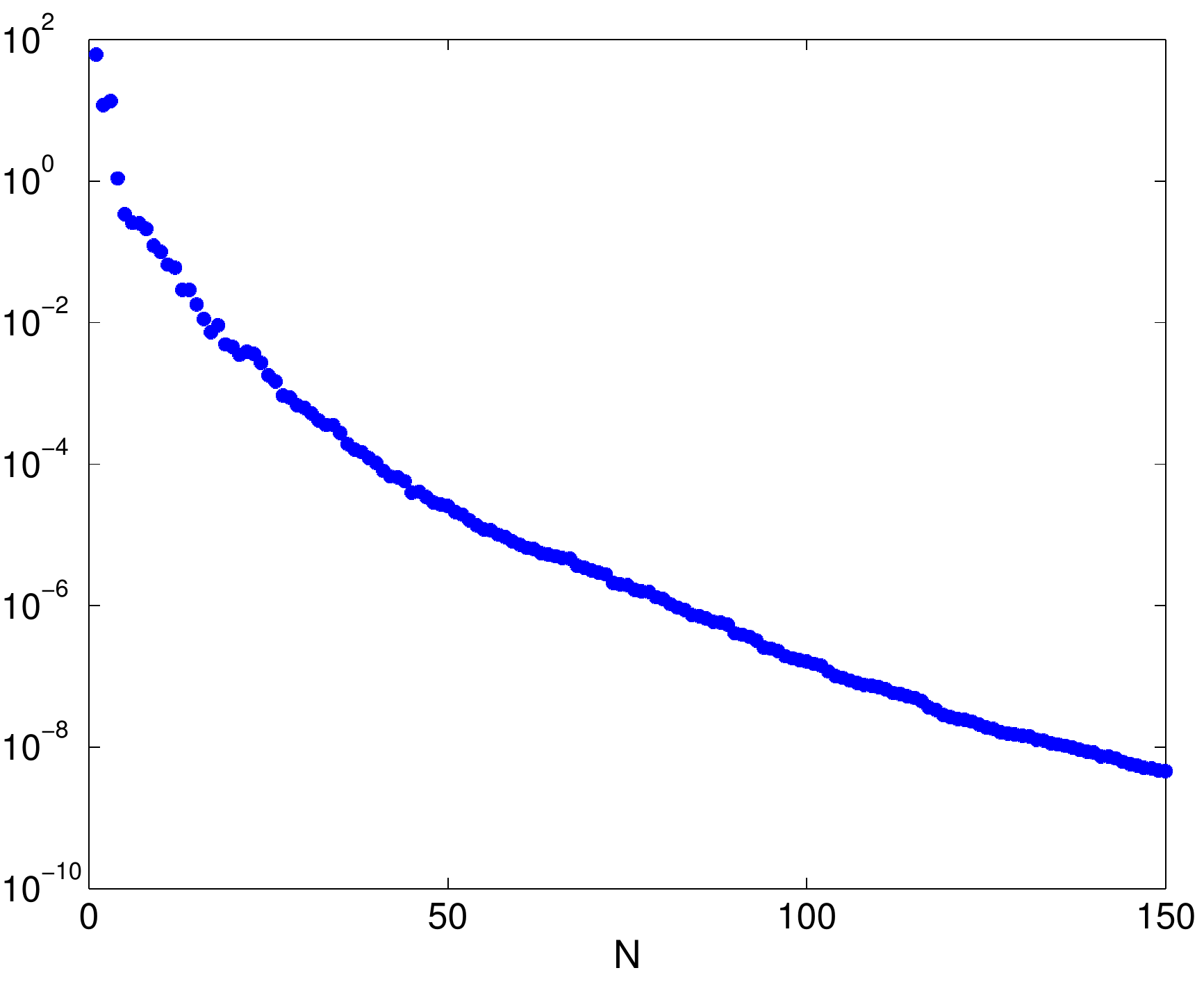}
 \includegraphics[width=0.5\linewidth]{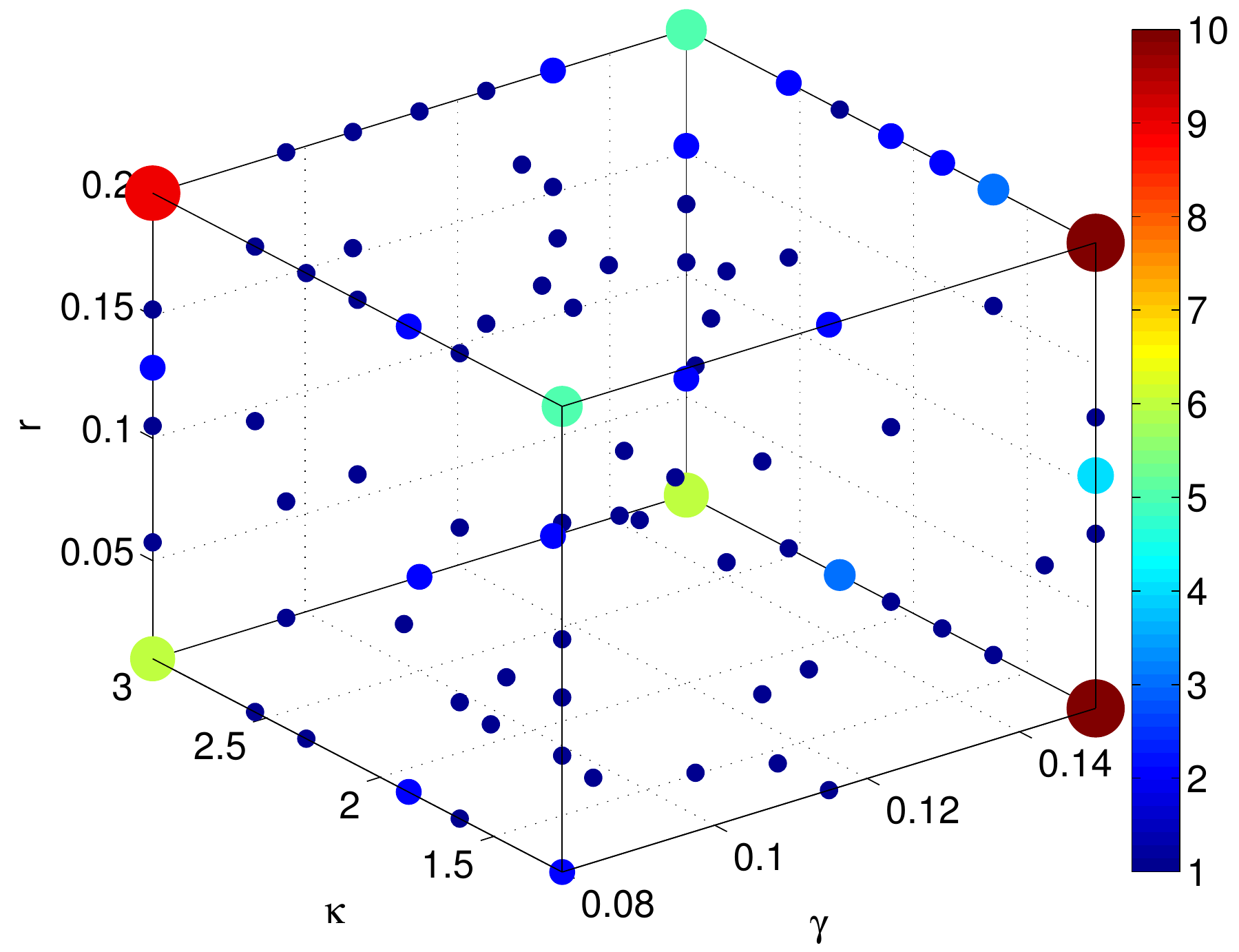}
 \caption{Left: Evolution of $\max_{\mu\in\mathcal{P}_{test}}\left\{E^{true}_{L^2}(\mu)\right\}$ for the European call option with the Heston model with $\mu=(\gamma,\kappa,r)$. The test grid is random $|{P}_{test}|=1024$. Right: Plot of the selected parameters $\mu_1,...,\mu_N\in\mathcal{P}_{train}$ and their frequency of the selection in the construction of the reduced basis in Algorithm~\ref{alg:global-greedy}. The train set composed of $|\mathcal{P}_{train}|=9^3=729$ equidistantly distributed points.}
\label{Fig:RB_3D_error}
\end{figure}
\begin{figure}
  \includegraphics[width=0.5\linewidth]{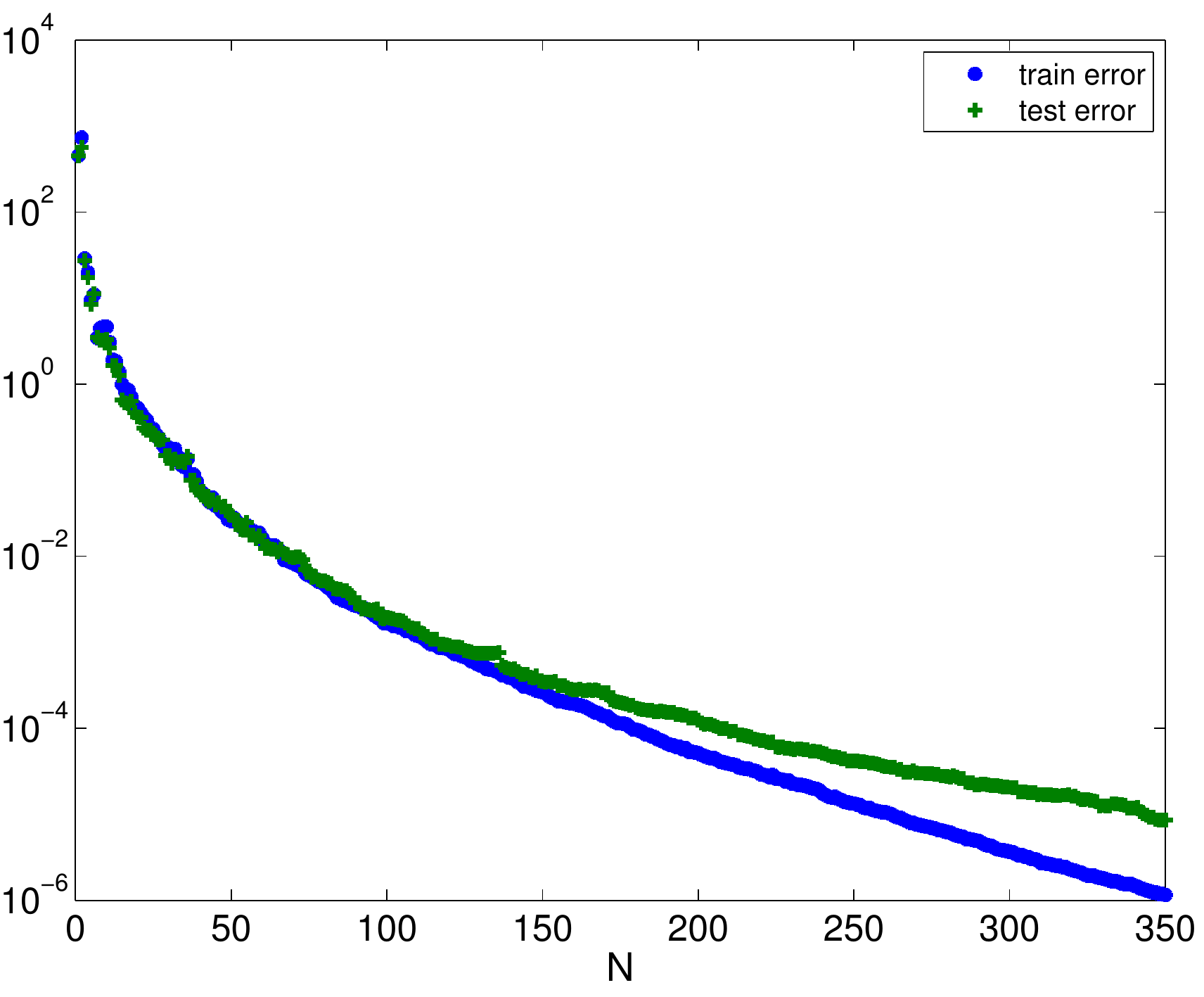}\centering
 \caption{Left: Evolution of the train error $\max_{\mu\in\mathcal{P}_{train}}\left\{E^{true}_{L^2}(\mu)\right\}$(blue stars) and a  test error $\max_{\mu\in\mathcal{P}_{test}}\left\{E^{true}_{L^2}(\mu)\right\}$(green crosses) for the European call option with the Heston model with $\mu=(\xi,\rho,\gamma,\kappa,r)$. The test grid is random $|{P}_{test}|=10000$. The train set composed of $|\mathcal{P}_{train}|=6^5=7776$ equidistantly distributed points.}
\label{Fig:RB_5D_error}
\end{figure}

\subsection{American options}

In this section, we present the numerical results corresponding to the performance of the reduced basis approach for pricing American options with the Black-Scholes and Heston model and the corresponding a posteriori error estimates. 
\subsubsection{Examples on the Heston model}\label{Sec:H_results}
For the American option case with the Heston model, we consider the settings of the parameter domain defined in~\eqref{eq:par_domHA}. For the experiments presented in this section, we consider $\mu\equiv(\gamma,\kappa)$ and, if not stated, the remaining parameter values are set to the corresponding entries of $\mu^*=(0.9, 0.21, 0.16, 3, 0.0198 )^T$.

The detailed primal solution and a corresponding Lagrange multiplier are presented in Figure~\ref{Fig:AO_exact}.  
\begin{figure} 
 \includegraphics[width=0.5\linewidth]{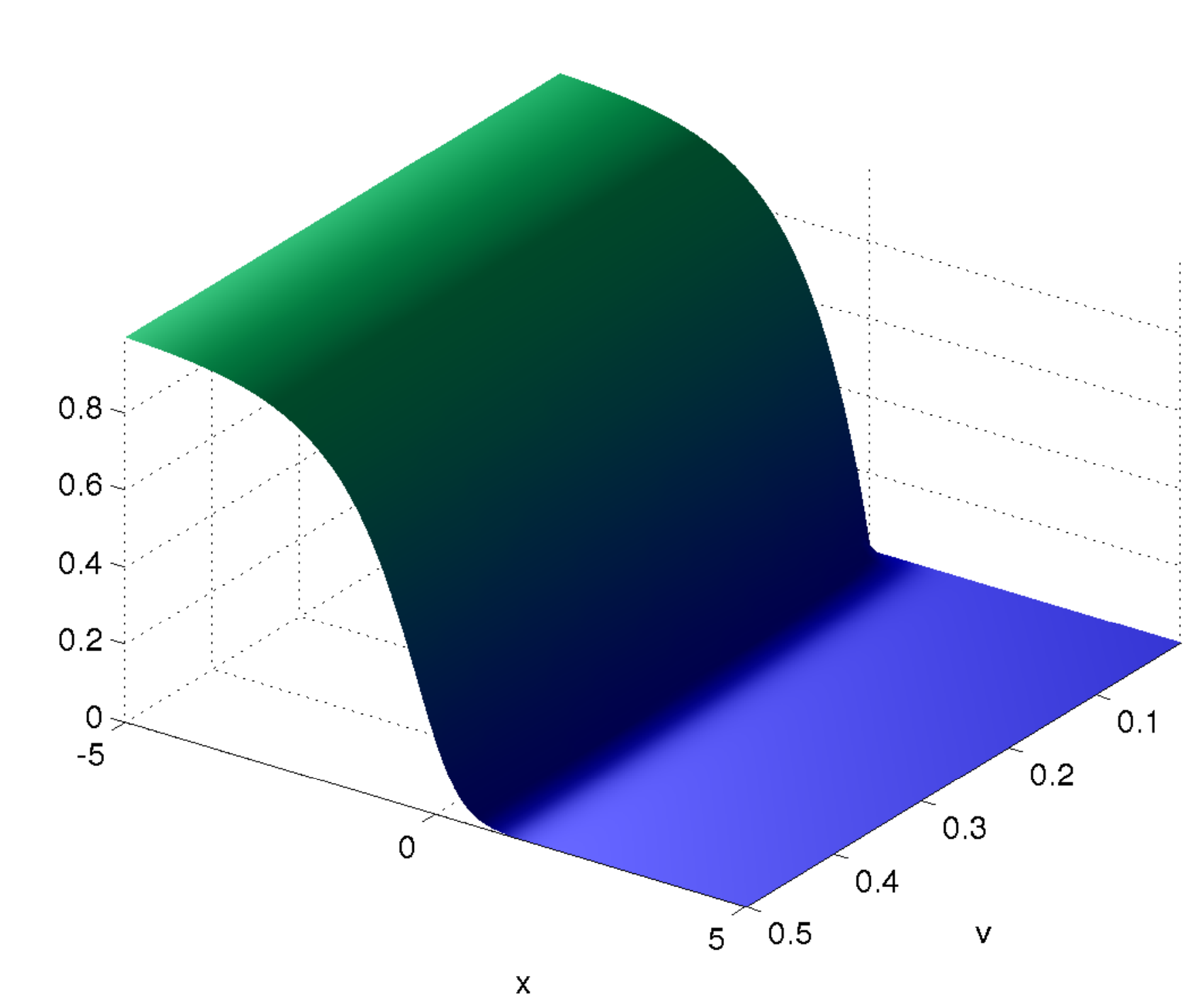}
 \includegraphics[width=0.5\linewidth]{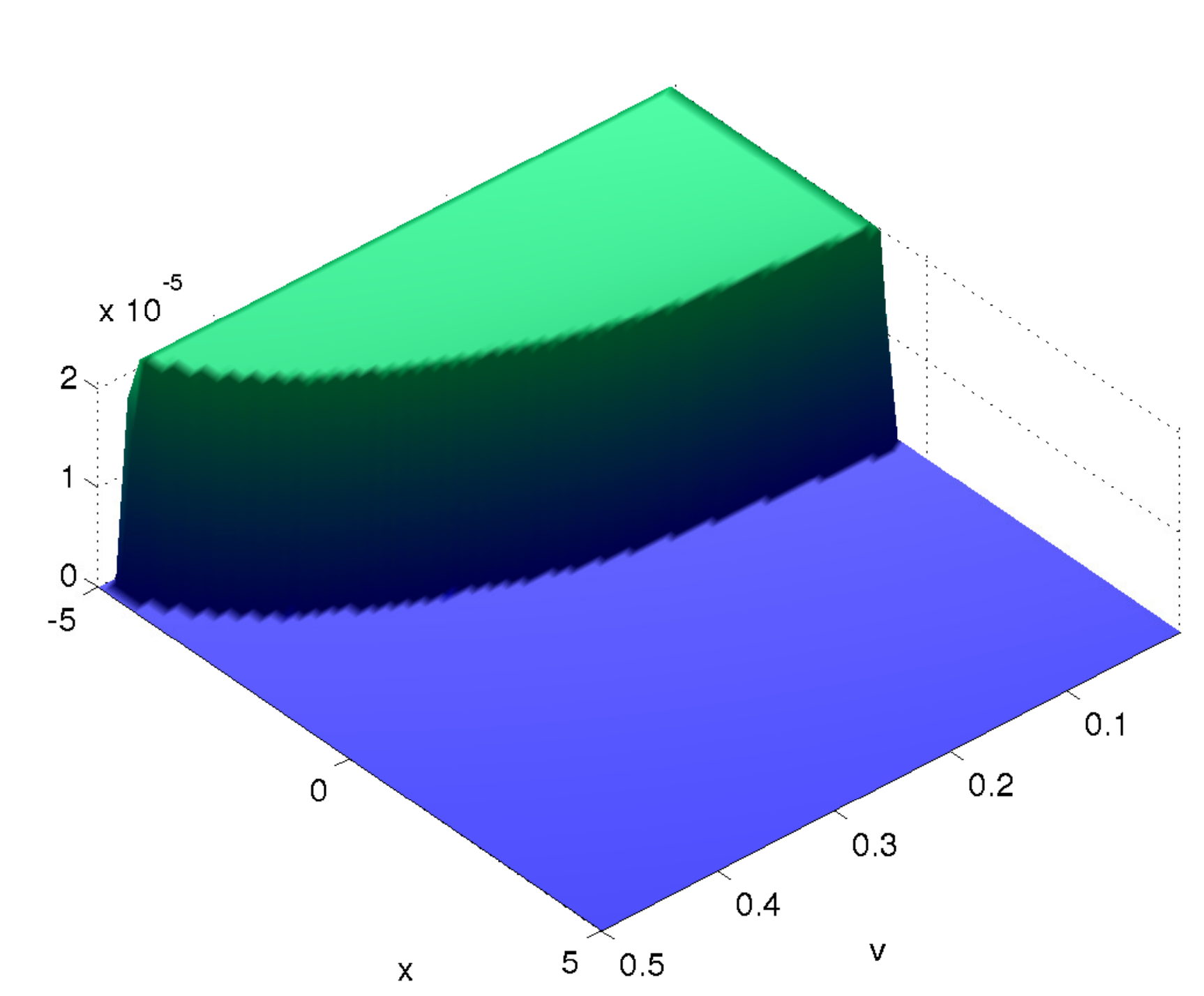}
 \caption{The solution of the American put with the Heston model for $\mu=(0.9, 0.21, 0.1750, 3, 0.0198)^T$ at $t=T$ (left) and a corresponding Lagrange multiplier (right).}
\label{Fig:AO_exact}
\end{figure}
In Figure~\ref{Fig:AO_snapshots}, we show the snapshots of the primal solutions and their Lagrange multipliers at different parameter values, which motivates us to apply the reduced basis method to this model.
\begin{figure} 
 \includegraphics[width=0.75\linewidth,height=7cm]{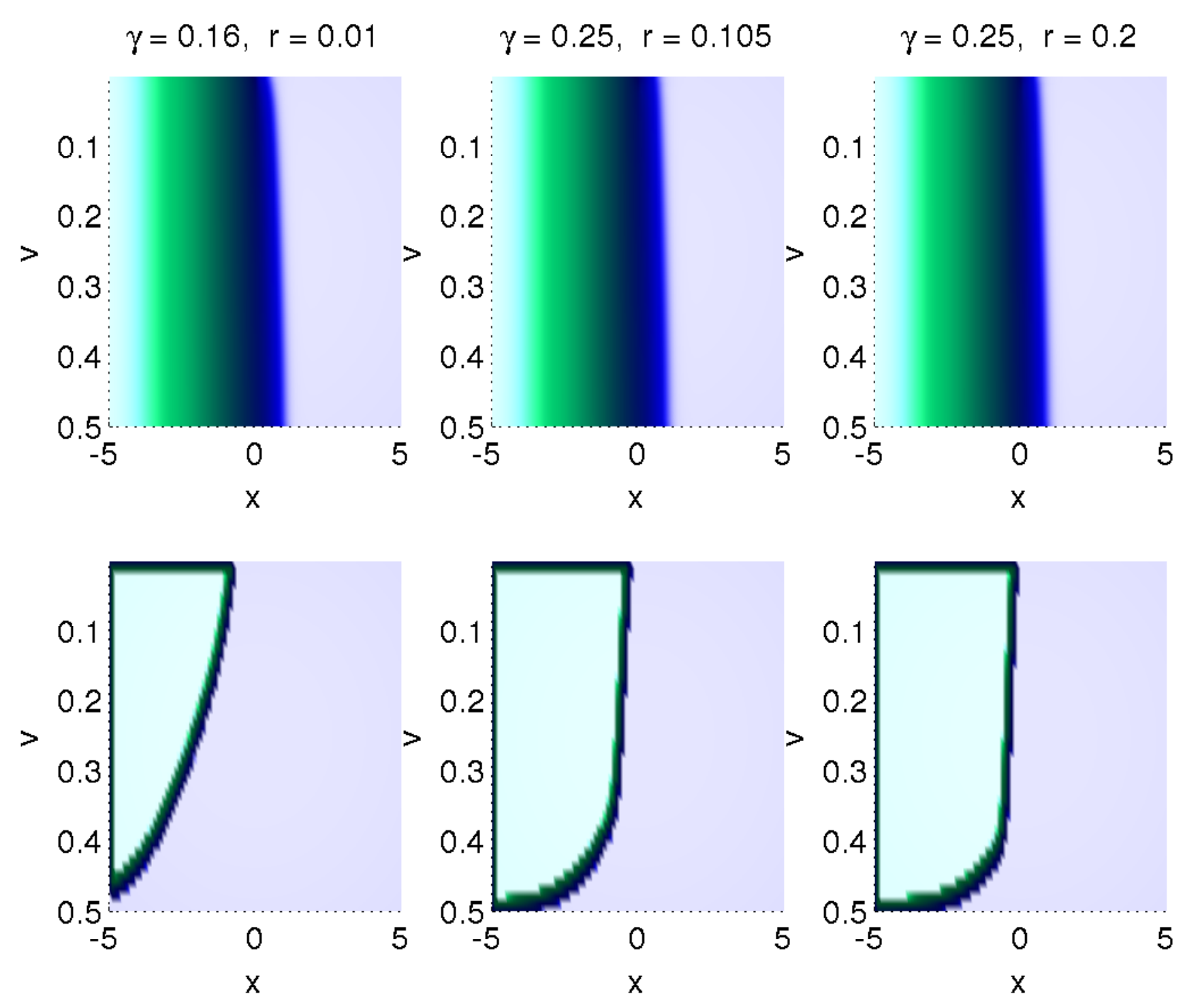}\centering
 \caption{Top: the primal solutions of the American put option with the Heston model with different values of $\gamma, r$ at the final time $t=T$ (top) and corresponding Lagrange multipliers (bottom).}
\label{Fig:AO_snapshots}
\end{figure}
To study the efficiency of the reduced basis approach, we consider the POD-Angle-Greedy algorithm with the true error indicator $E(\mu)=E_{L^2}^{true}(\mu)$. We build the hierarchical reduced basis with $N_V=70$ and $N_W=35$ using the training set $|\mathcal{P}_{train}|=7^2=49$ uniformly distributed points. The evolution of the error produced by the algorithm is depicted in Figure~\ref{Fig:AO_basis2}. Each reduced basis model is tested on the random test set $|\mathcal{P}_{test}|=200$ and the corresponding error $\max_{\mu\in\mathcal{P}_{test}}E_{L^2}^{true}(\mu)$ is also depicted in Figure~\ref{Fig:AO_basis2}. 
\begin{figure} 
 \includegraphics[width=0.5\linewidth]{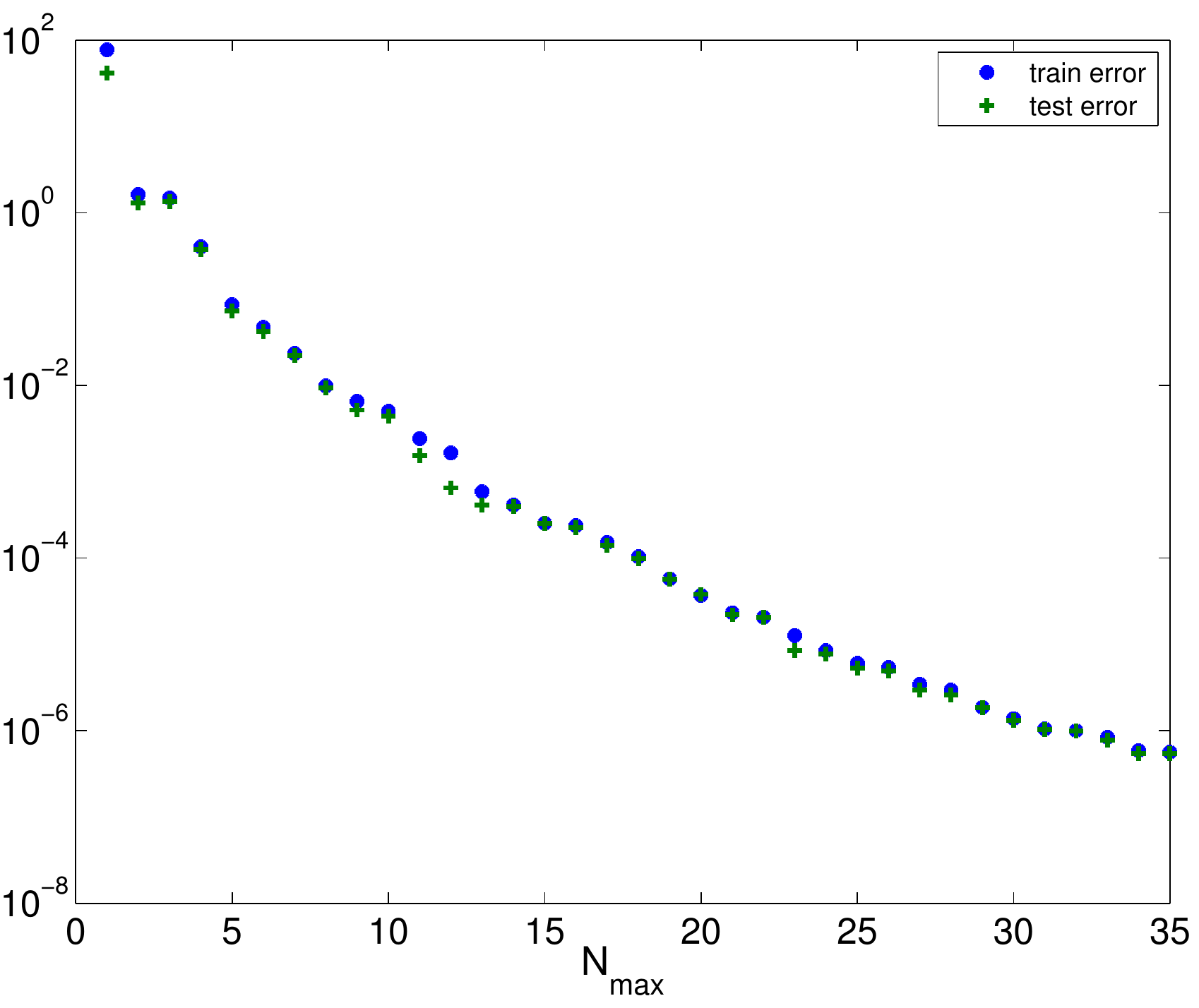}\centering
 \caption{Evolution of the  train error $\max_{\mu\in\mathcal{P}_{train}}E^{true}_{L^2}(\mu)$ during the iterations of Algorithm~\ref{alg:global-greedy} (green stars) and the test error $\max_{\mu\in\mathcal{P}_{test}}E^{true}_{L^2}(\mu)$ (blue crosses) for the American put with the Heston model.  }
\label{Fig:AO_basis2}
\end{figure}
Similar to the European option case, we observe a good approximation property of the reduced basis method and an exponential convergence of the error.
\subsubsection{Examples on the Black-Scholes model}\label{Sec:BS_results}
Now, we restrict ourselves to the Black-Scholes model. We first demonstrate the parameter and time dependence of our model. In Figure~\ref{Fig:BS_snapshots}, the examples of the primal and dual solutions at different parameter values $\sigma$ and time steps  are presented. 
\begin{figure}[!h]
\centering
 \includegraphics[width=0.32\linewidth]{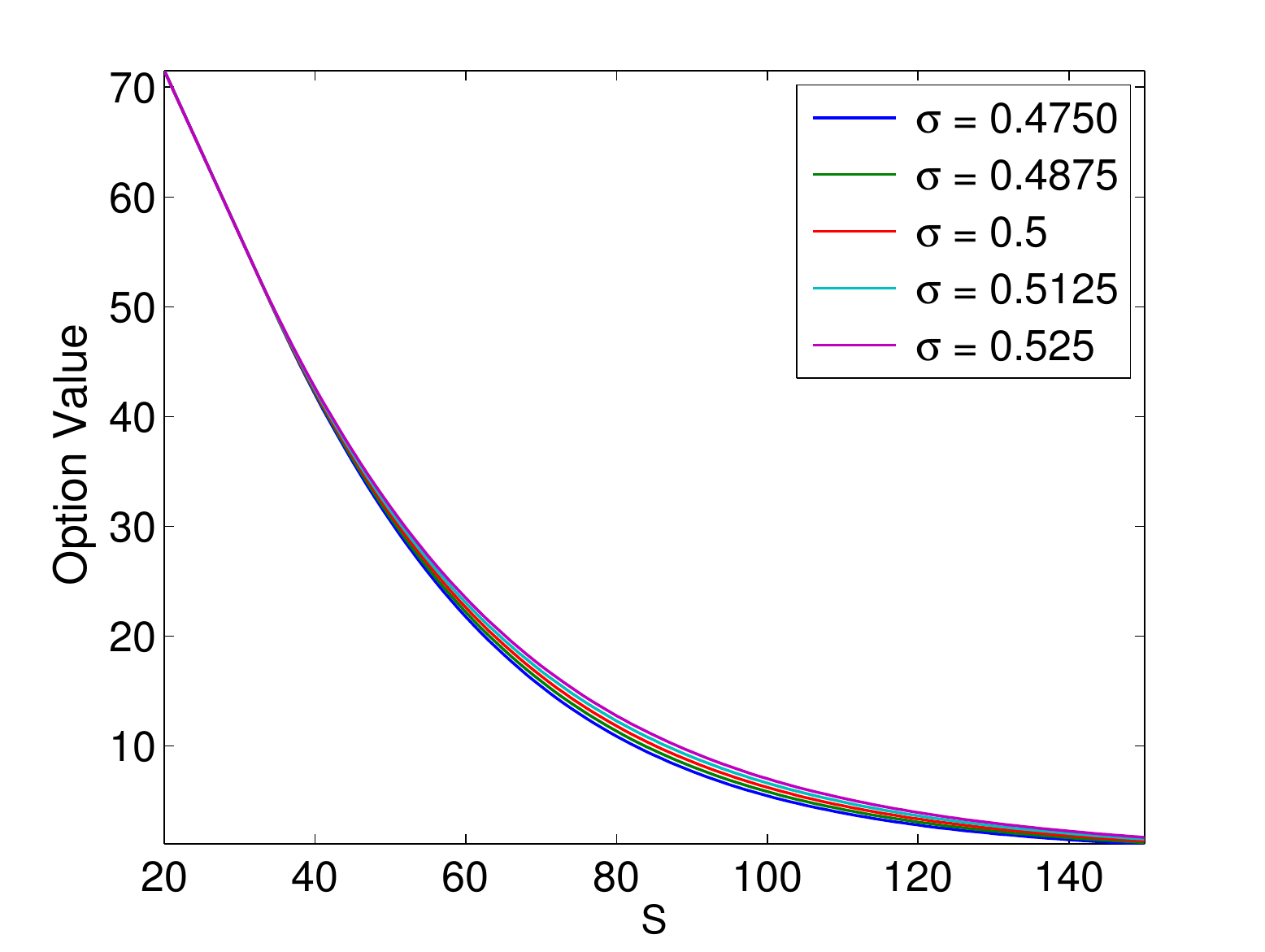}
  \includegraphics[width=0.32\linewidth]{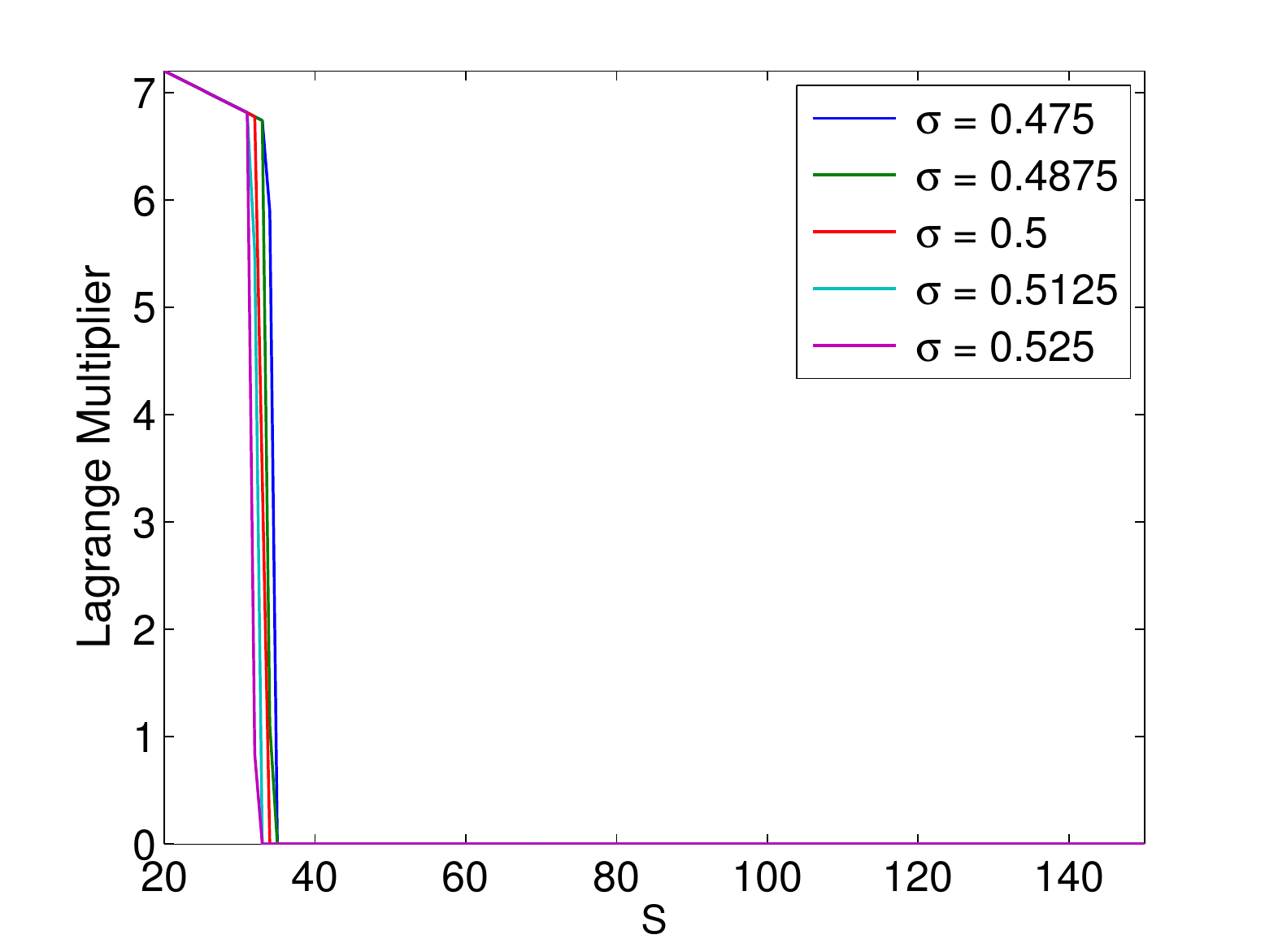}
   \includegraphics[width=0.32\linewidth]{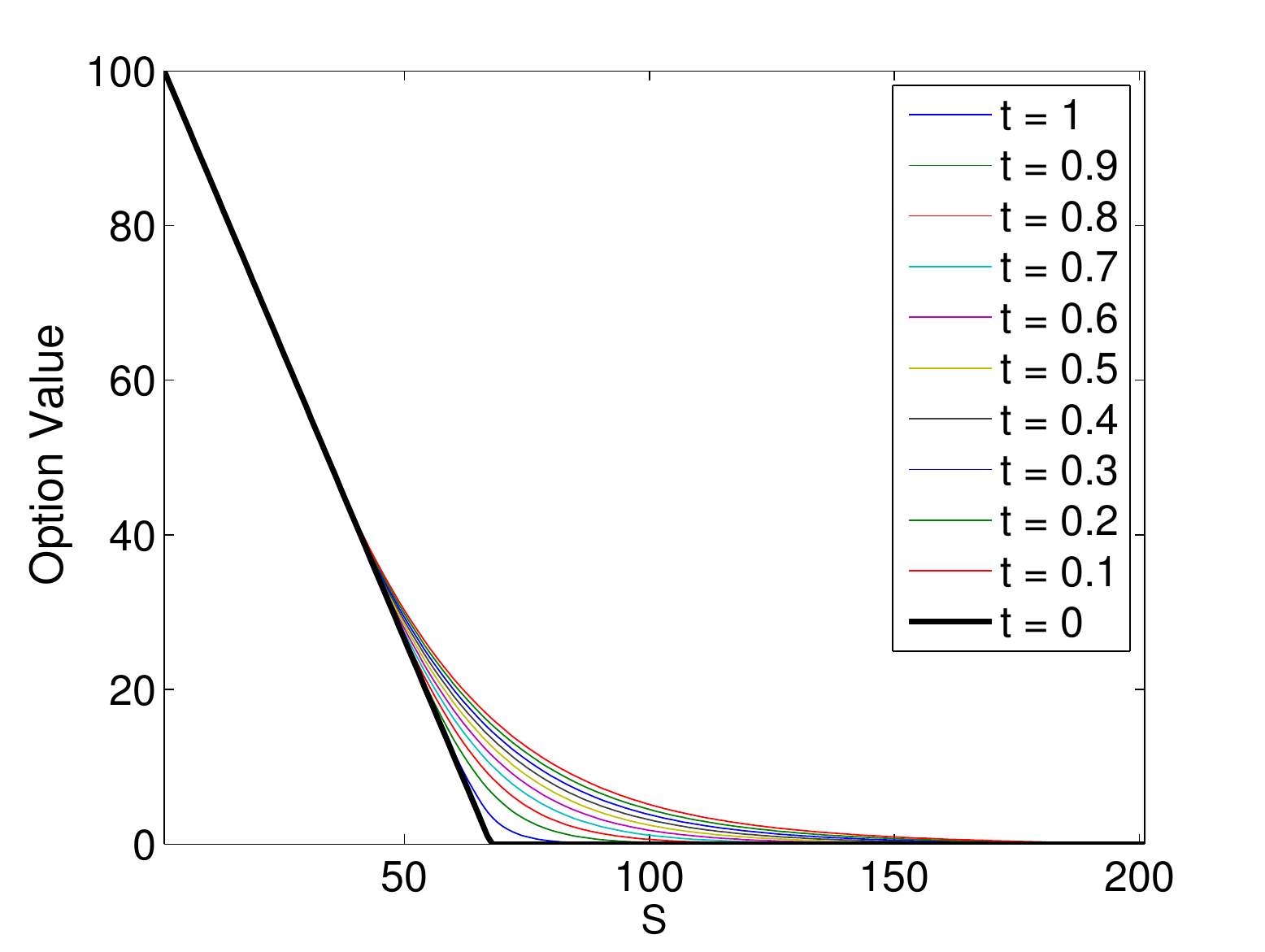}
 \caption{The primal solutions (left) and corresponding Lagrange multipliers (middle) of the American put option with the Black-Scholes model with different values of the parameter $\sigma$ at a final time step $t=T$. The evolution of the solution for different times (right). }
\label{Fig:BS_snapshots}
\end{figure}

As the next set of experiments, we consider the performance of the reduced basis approach using Algorithm~\ref{alg:global-greedy} with the true energy error  $E(\mu)=E_{energy}^{true}(\mu)$. 
We construct the bases using the proposed  algorithm with $N_V=50$, $N_W=25$ and test the reduced basis approach first for a random value of the parameter $\mu=(4.8470\cdot 10^{-2}, 7.6785 \cdot 10^{-3}, 4.1856 \cdot 10^{-1})\in{\cal P}\setminus{\cal P}_{train}$. Some steps of the simulation are represented in
Figure~\ref{fig:simu}. The reduced basis and fine detailed simulation curves are hardly distinguishable, that reveals a good reduced basis approximation.
\begin{figure}[!h]
\centering
\includegraphics[width=0.31\linewidth]{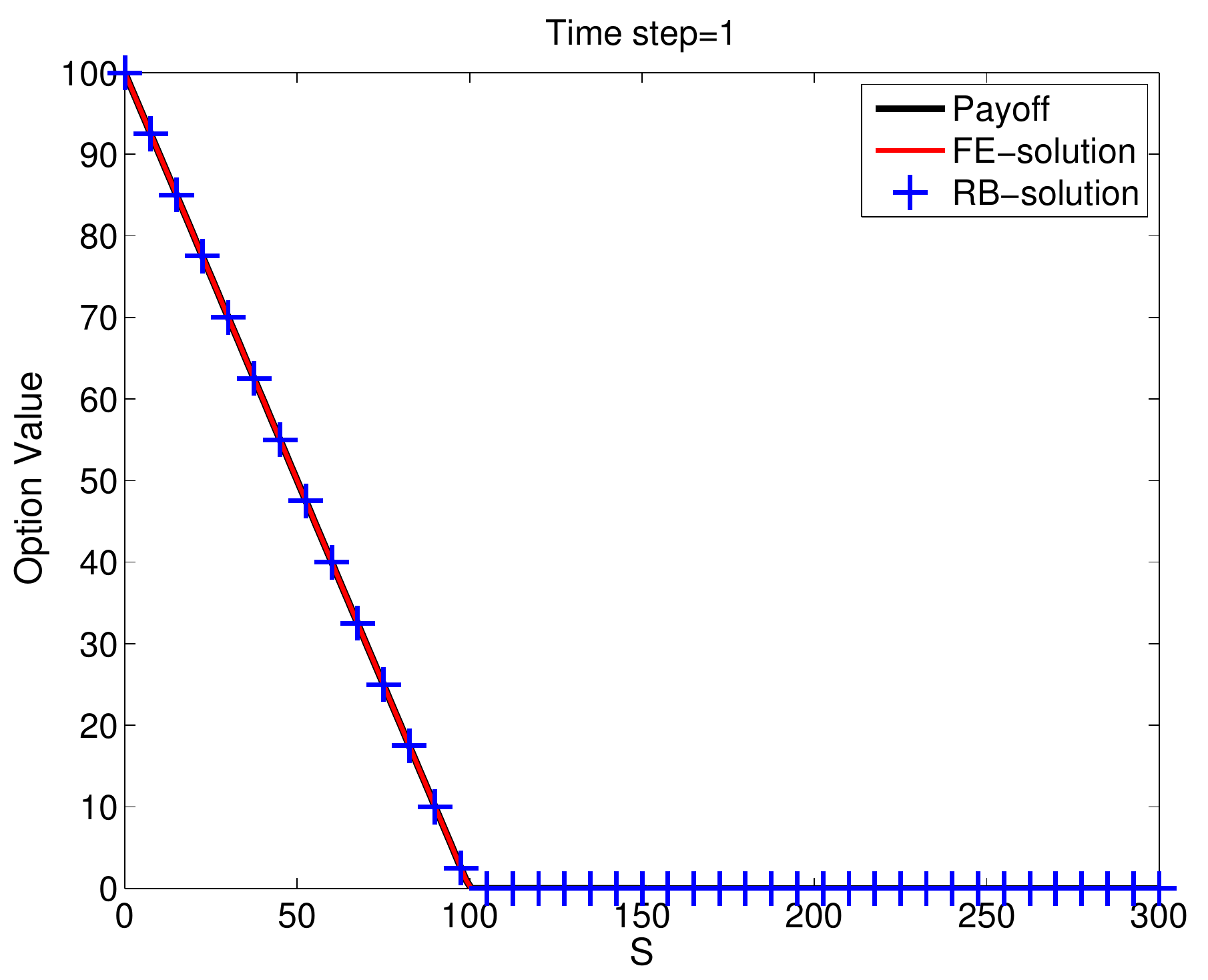}
\includegraphics[width=0.31\linewidth]{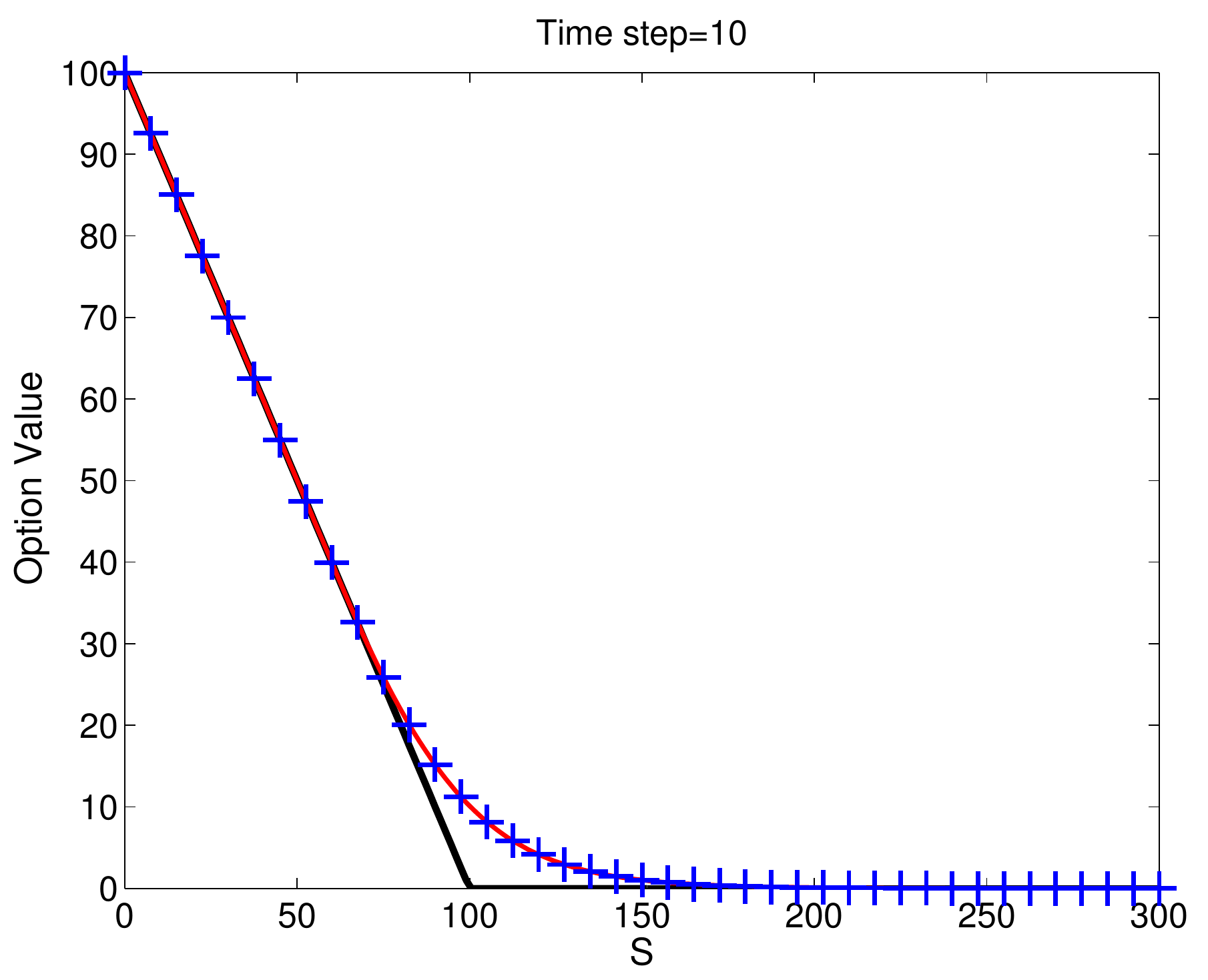}
\includegraphics[width=0.31\linewidth]{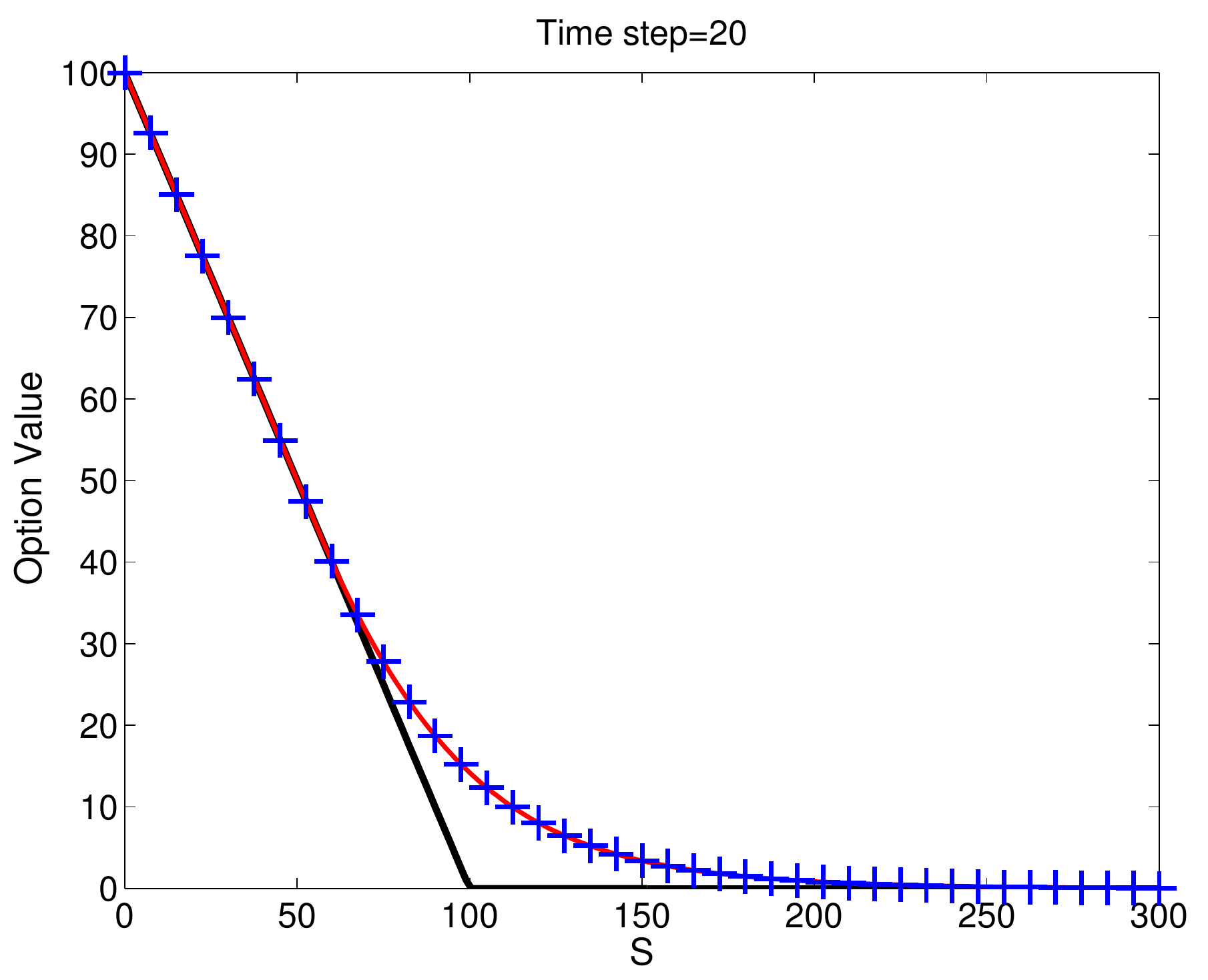}
\caption{A finite element (solid red line) and a reduced basis
  approximation (blue $+$) at time steps $t/\Delta t=1$, $t/\Delta
  t=10$ and $t/\Delta t=T/\Delta t=20$ for the American put option with the Black-Scholes model, with $\mu=(4.8470\cdot 10^{-2}, 7.6785 \cdot 10^{-3}, 4.1856 \cdot 10^{-1})^T$. The payoff function is represented as the black line.}
\label{fig:simu}
\end{figure}

We then test our algorithm on a larger set of parameters.
We consider ${\cal P}_{test}\subset {\cal P}$, a random set of
$|\mathcal{P}_{test}|=20$ parameters and estimate the mean value of the error $E_{energy}^{true}(\mu)$ over ${\cal P}_{test}$.
More precisely, we evaluate $\frac{1}{|{\cal P}_{test}|}\sum_{\mu\in{\cal P}_{test}}E_{energy}^{true}(\mu)$.  
The results are plotted in the diagram presented in Figure~\ref{fig:ErrorCurve}. 
We observe an error decay of several orders in magnitude when simultaneously
increasing $N_V$ and $N_W$. 
\begin{figure}
 \includegraphics[width=0.5\linewidth]{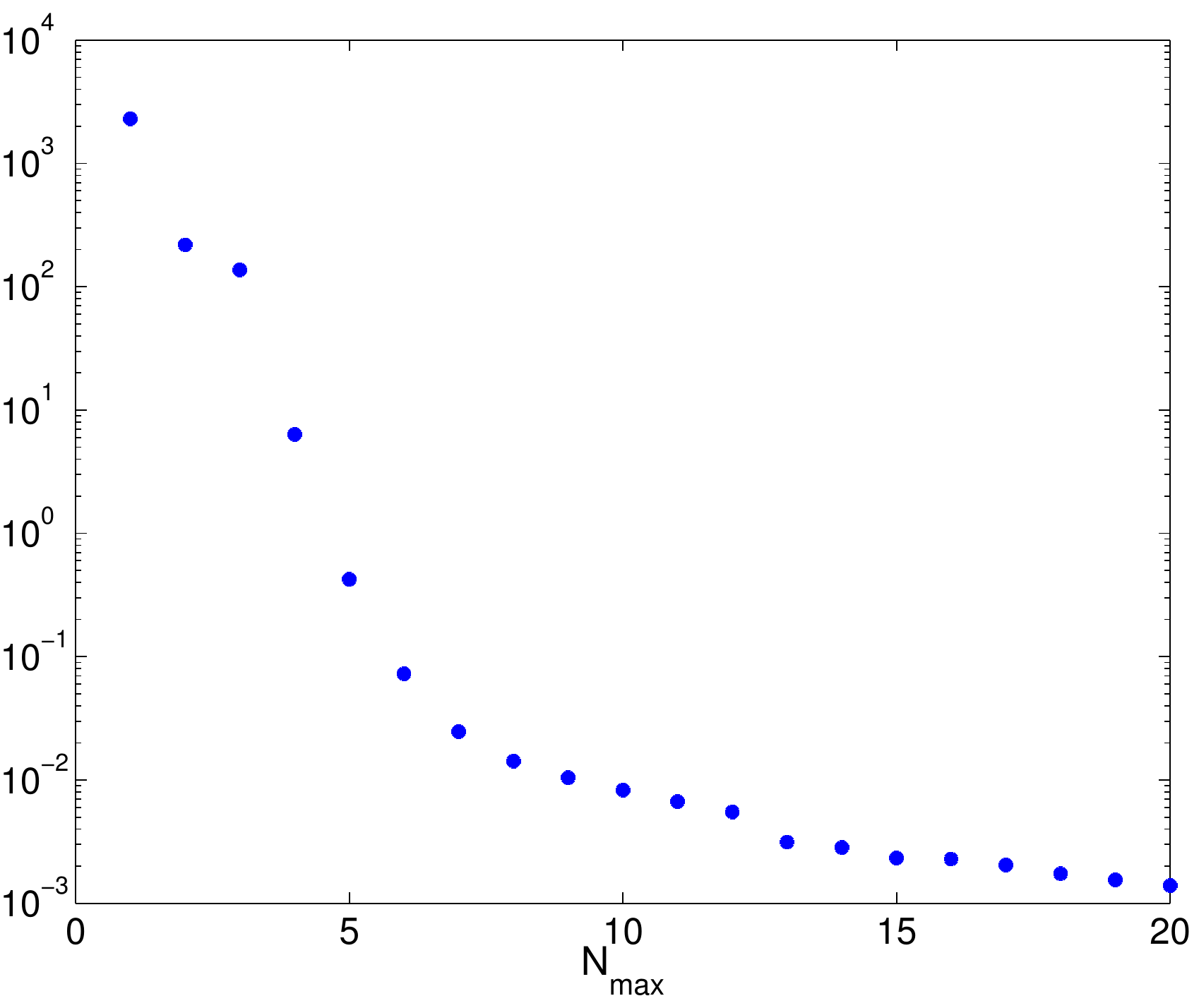}\centering
\caption{Values of $\frac{1}{|{\cal P}_{test}|}\sum_{\mu\in{\cal P}_{test}}E_{energy}^{true}(\mu)$ when using various hierarchical bases obtained with Algorithm~\ref{alg:global-greedy} for the American put option with the Black-Scholes model.}
\label{fig:ErrorCurve}
\end{figure}
\subsubsection{Efficiency of a posteriori error estimates}
 We now focus on the efficiency of the design procedure and an inclusion of an a posteriori error estimate using as an example the American put option with the Black-Scholes and Heston model. For this, we consider the POD-Angle-Greedy algorithm and use alternatively $E(\mu)=E_{energy}^{true}(\mu)$ and $E(\mu)=E_{energy}^{Apost}(\mu)$ as selection criterion with $N_{\max}=25$, i.e., $N_V=50$, $N_W=25$ for both models. The other parameter settings remain the same as introduced in Section~\ref{Sec:BS_results} and Section~\ref{Sec:H_results}.
The visualization of the reduced basis vectors of $\Psi_N$, $\Xi_N$ for the Black-Scholes model with $E(\mu)=E_{energy}^{Apost}(\mu)$ is
represented in Figure~\ref{Fig:Basis_BS}. 
\begin{figure}
\includegraphics[width=.5\linewidth]{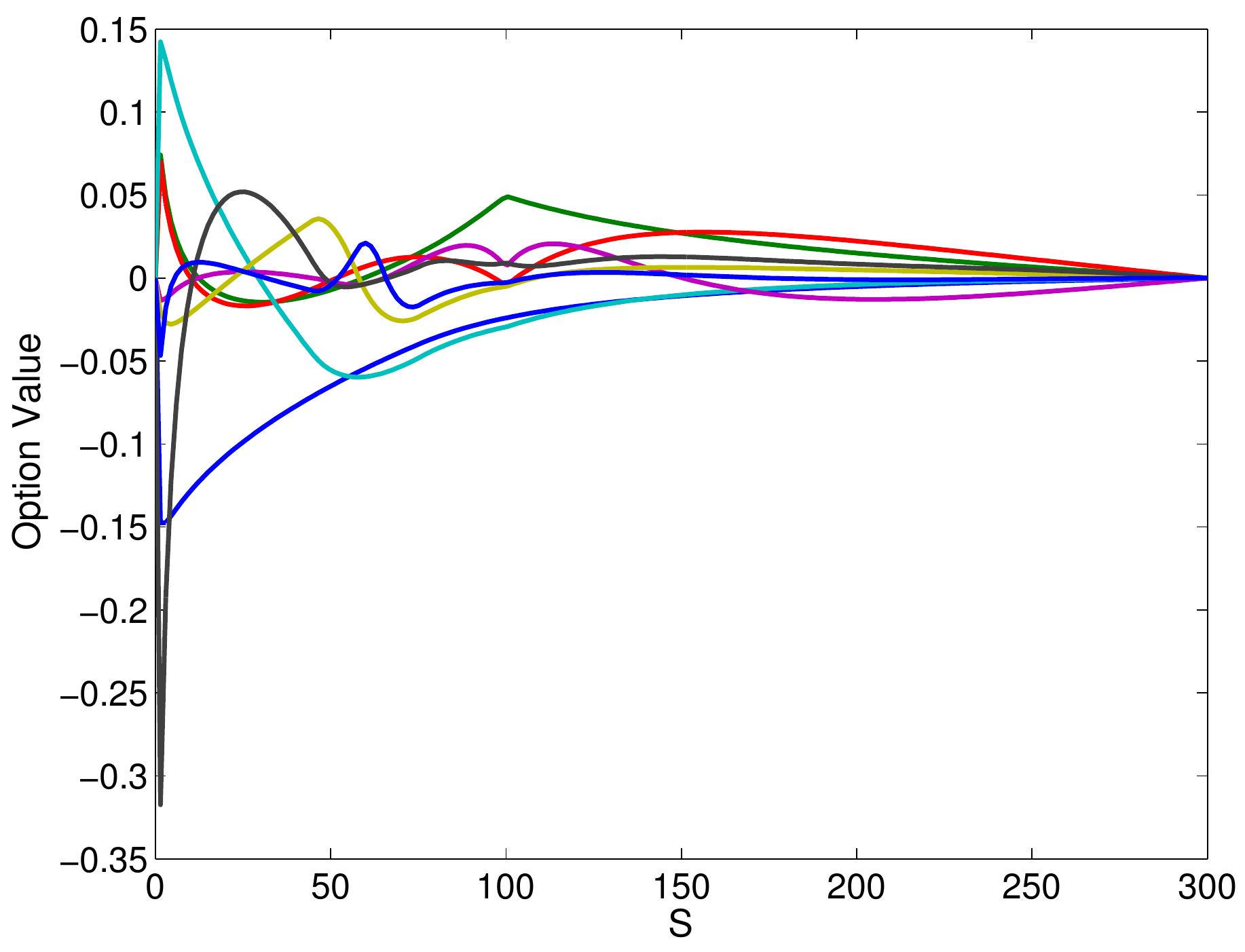}
\includegraphics[width=.5\linewidth]{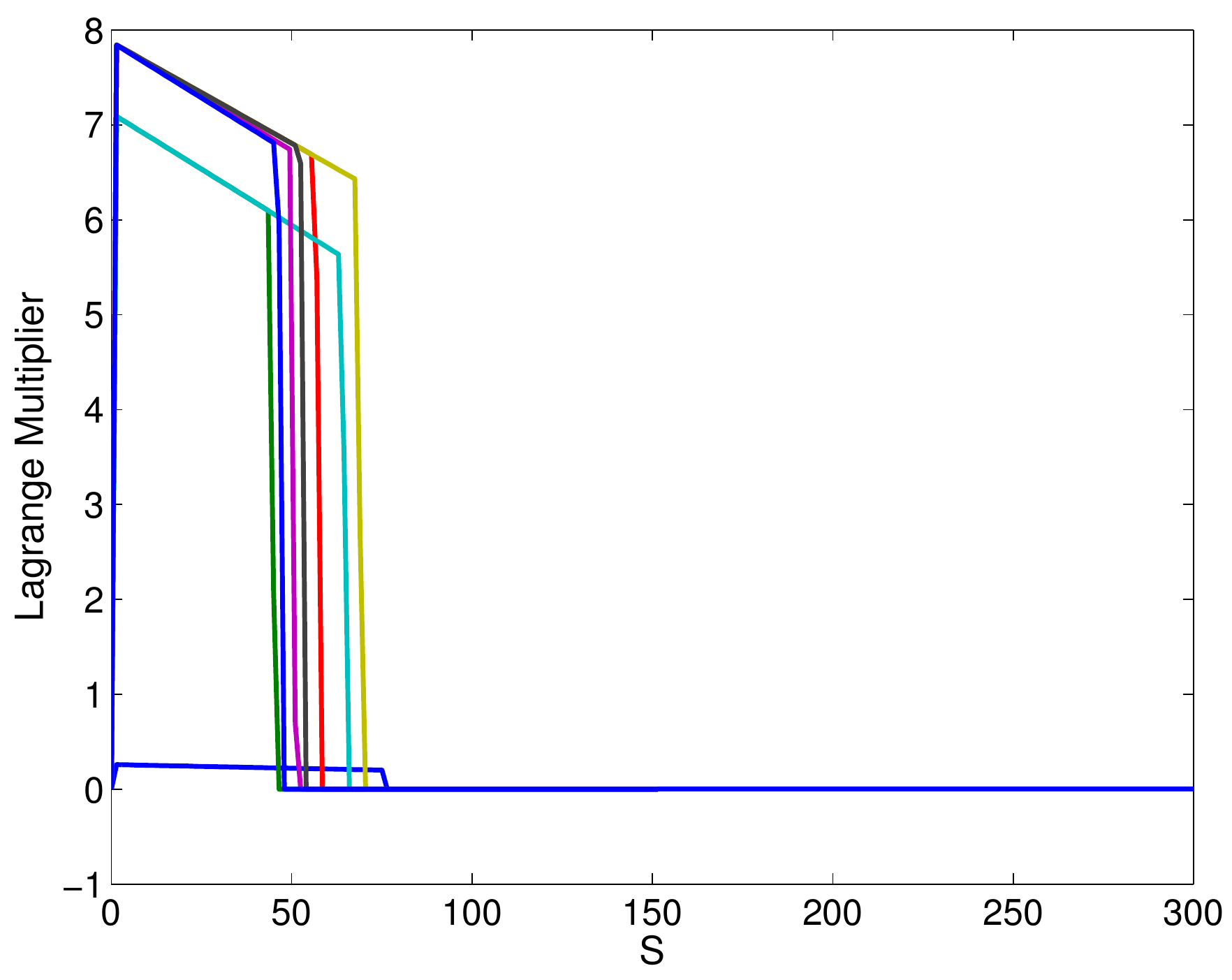}
\caption{The eight first vectors of the primal (left) and dual (right) bases obtained with Algorithm~\ref{alg:global-greedy} with $E(\mu)=E_{energy}^{Apost}(\mu)$ for the American put option with the Black-Scholes model.}
\label{Fig:Basis_BS}
\end{figure}

We now comment on an employment of the a posteriori error bounds developed in  Section~\ref{sec:error_analysis}. We consider the values of the quantity $\max_{\mu\in\mathcal{P}_{train}}E(\mu)$ 
along the iterations of Algorithm~\ref{alg:global-greedy}, choosing either 
$E(\mu)=E_{energy}^{true}(\mu)$ or $E(\mu)=E_{energy}^{Apost}(\mu)$. The corresponding results for the Black-Scholes model are presented in Figure~\ref{Fig:convGreedy} and for the Heston model in Figure~\ref{Fig:convGreedyH}.
We observe that for both models using $E(\mu)=E_{energy}^{true}(\mu)$ as an error measure in Algorithm~\ref{alg:global-greedy} results in less monotone error convergence of $E_{energy}^{Apost}(\mu)$ with respect to $N_{\max}$ and vice versa. Overall we obtain similar accuracy when using the (cheap) $E_{energy}^{Apost}(\mu)$ measure in contrast to the expensive true error $E_{energy}^{true}(\mu)$, which illustrates the relevance of our a posteriori analysis.
\begin{figure}
\includegraphics[width=0.5\linewidth]{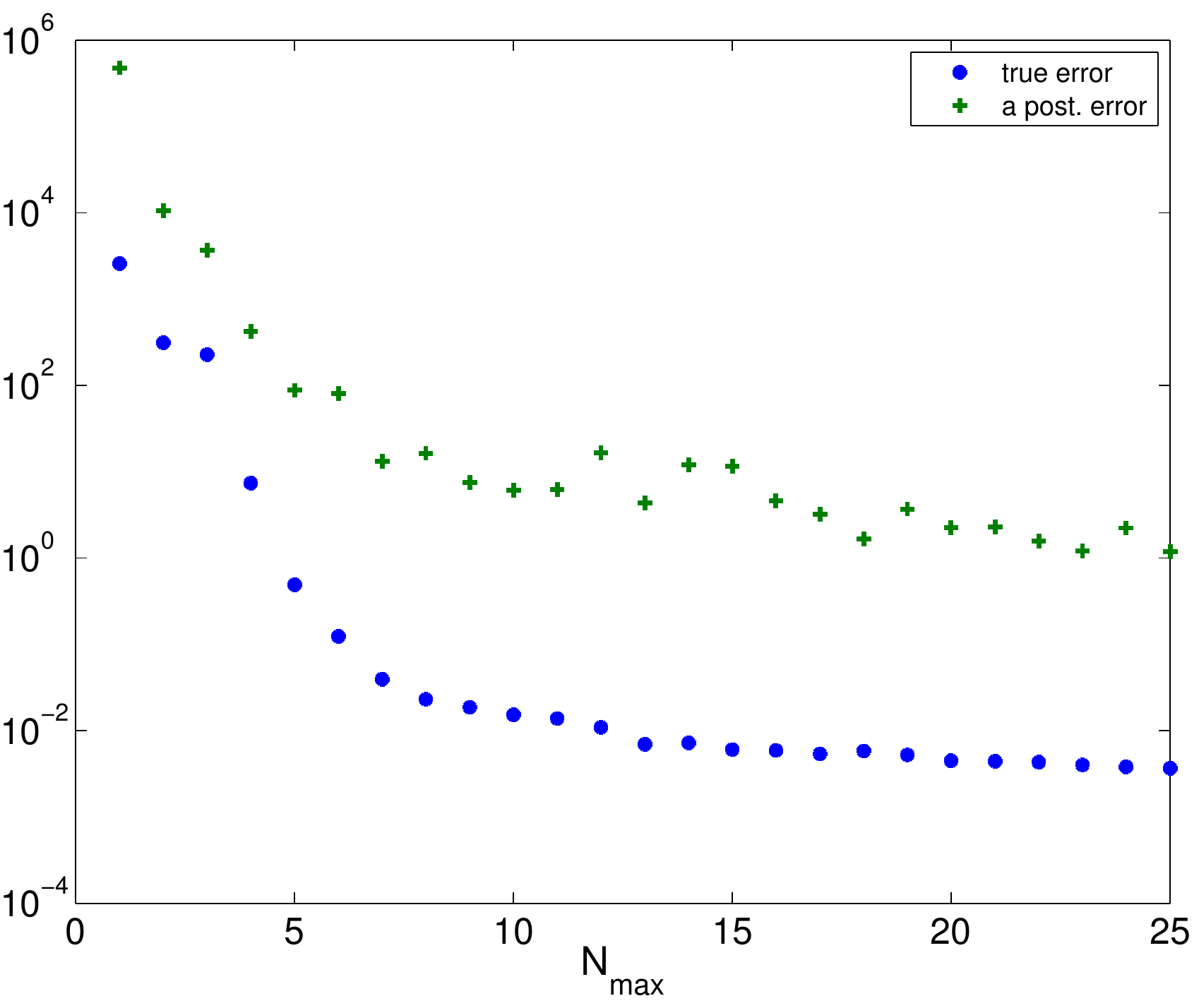}
\includegraphics[width=0.5\linewidth]{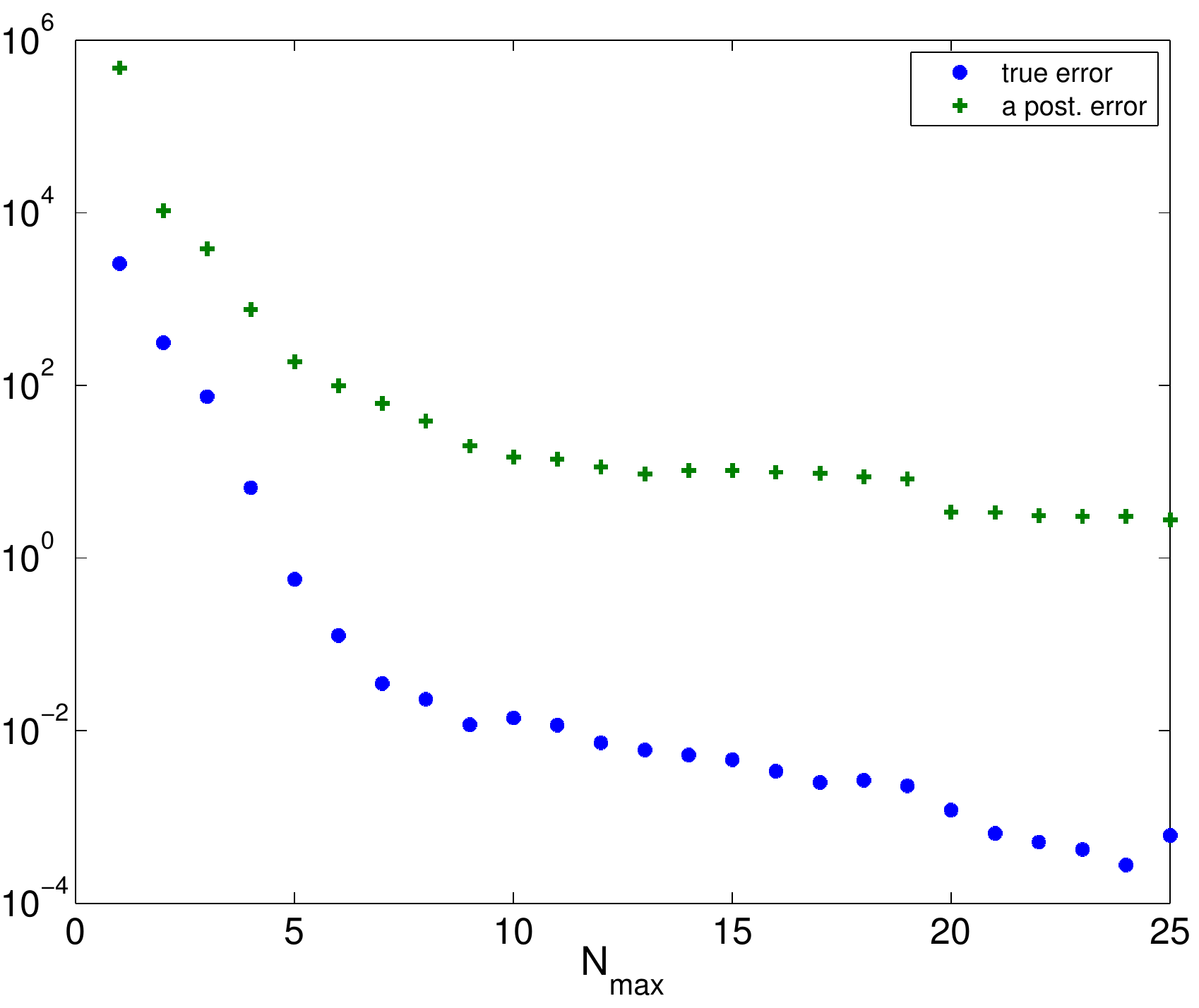}
\caption{Evolution of the train error $\max_{\mu\in\mathcal{P}_{train}}E(\mu)$ for the American put option with the Black-Scholes model during the iterations of Algorithm~\ref{alg:global-greedy}  with error estimator $E(\mu)=E_{energy}^{true}(\mu)$ (left) and with $E(\mu)=E_{energy}^{Apost}(\mu)$ (right). Blue stars: values of $E_{energy}^{true}(\mu)$, green crosses: values of $E_{energy}^{Apost}(\mu)$.}\label{Fig:convGreedy}
\end{figure}
\begin{figure}
\includegraphics[width=0.5\linewidth]{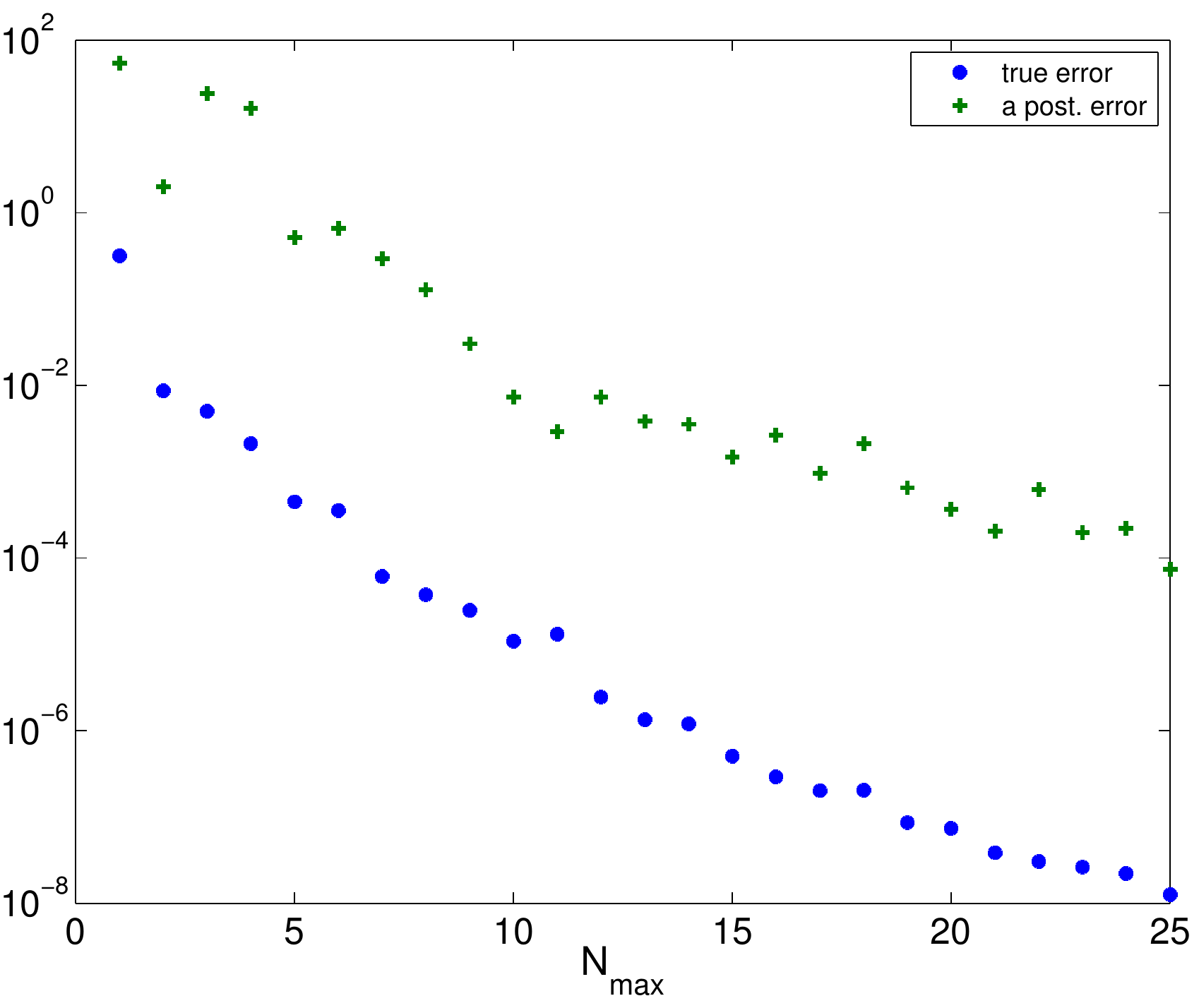}
\includegraphics[width=0.5\linewidth]{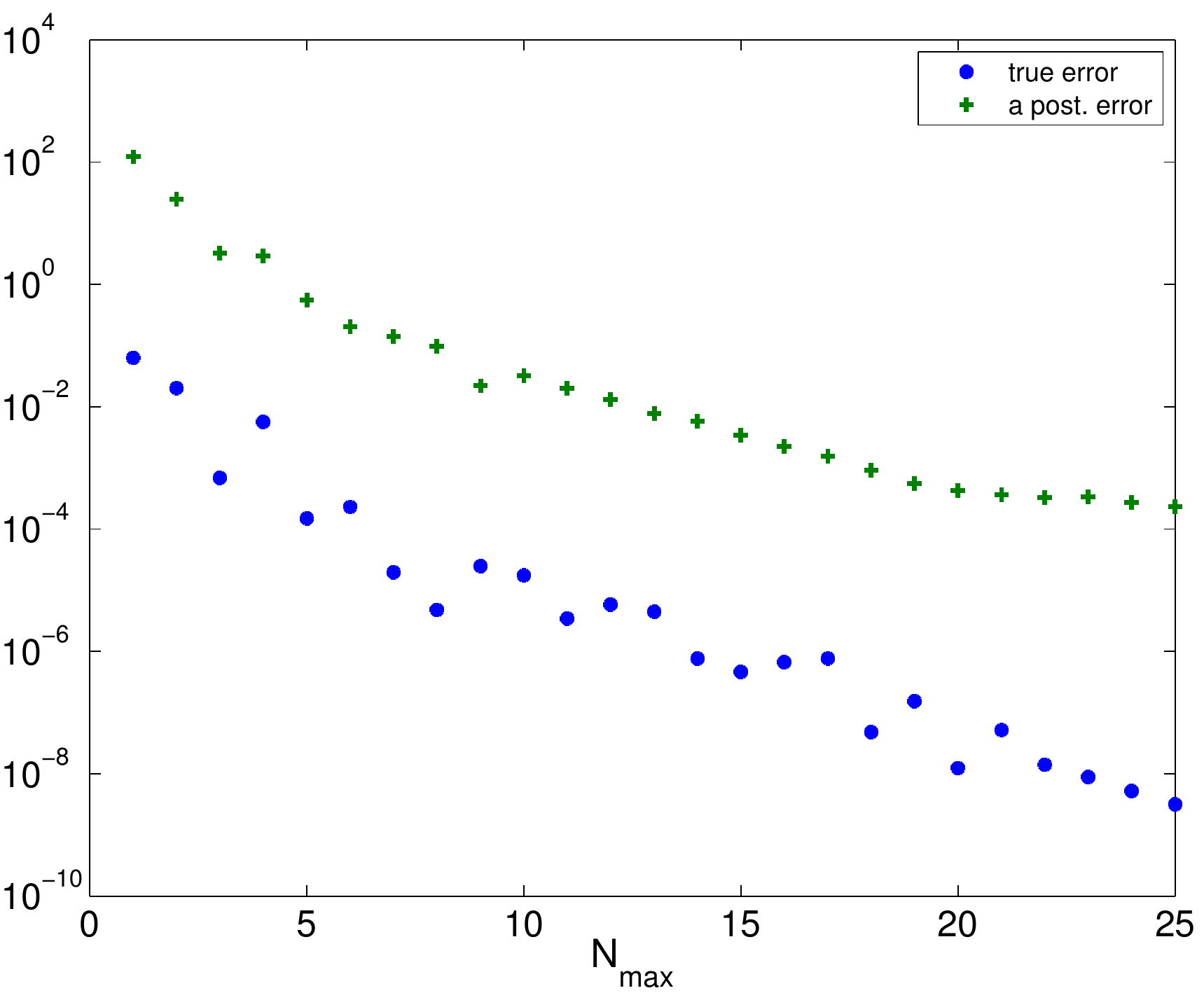}
\caption{Evolution of the train error $\max_{\mu\in\mathcal{P}_{train}}E(\mu)$ for the American put option with the Heston model during the iterations of Algorithm~\ref{alg:global-greedy} with error estimator $E(\mu)=E_{energy}^{true}(\mu)$ (left) and with $E(\mu)=E_{energy}^{Apost}(\mu)$ (right). Blue stars: values of $E_{energy}^{true}(\mu)$, green crosses: values of $E_{energy}^{Apost}(\mu)$.}\label{Fig:convGreedyH}
\end{figure}

As a measure of the quality of the proposed error estimate, we define the associated effectivities
\begin{equation}
 \eta_N(\mu)=\sqrt{\frac{E_{energy}^{Apost}(\mu)}{E_{energy}^{true}(\mu)}}.
\end{equation}
Table~\ref{Table:BS} and Table~\ref{Table:Heston} provide the maximum effectivities $\max_{\mu\in\mathcal{P}_{train}}\eta_N(\mu)$ associated with the error bounds for the Black-Scholes and Heston model using different error measures in the POD-Angle-Greedy algorithm. Note, that the choice of $E(\mu)$ in the algorithm does not have a significant impact on the values of the maximum effectivities. We also observe that for the Heston model the effectivity values are higher than for the Black-Scholes one, which can be justified by a more complex nature of the model. Overall the effectivities of two orders of magnitude are very well acceptable for instationary RB problems.
\begin{table}[ht]
\centering
The choice of $E(\mu)=E_{energy}^{Apost}(\mu)$\\[4pt]
\begin{tabular}{|c|c|c|c|c|c|}                         
\hline                                                 
  $N_{\max}$ & $(N_V, N_W)$ & $n = 5$ & $n = 10$ & $n = 15$ & $n = 20$ \\   
\hline     \hline                                            
4 & $(8, 4)$ & 4.9e+00 & 5.0e+00 & 6.2e+00 & 1.1e+01 \\ 
\hline                                                 
8 & $(16, 8)$ & 4.5e+01 & 4.3e+01 & 4.3e+01 & 4.2e+01 \\ 
\hline                                                 
16 & $(32, 16)$ & 6.4e+01 & 6.2e+01 & 6.4e+01 & 6.5e+01 \\
\hline                                                 
20 & $(40, 20)$ & 1.2e+01 & 7.1e+01 & 7.9e+01 & 8.5e+01 \\
\hline                                                 
24 & $(48, 24)$ & 6.4e+01 & 6.7e+01 & 7.4e+01 & 1.0e+02 \\
\hline                                                 
\end{tabular}                                                                                    
\\[8pt]
The choice of $E(\mu)=E_{energy}^{true}(\mu)$\\[4pt]
\centering 
\begin{tabular}{|c|c|c|c|c|c|}                         
\hline                                                 
 $N_{\max}$ & $(N_V, N_W)$ & $n = 5$ & $n = 10$ & $n = 15$ & $n = 20$ \\   
\hline                                                 
4 & $(8, 4)$ & 5.4e+00 & 6.1e+00 & 7.2e+00 & 1.0e+01 \\ 
\hline                                                 
8 & $(16, 8)$ & 5.1e+01 & 5.3e+01 & 5.0e+01 & 4.8e+01 \\ 
\hline                                                 
16 & $(32, 16)$ & 7.3e+01 & 7.0e+01 & 7.0e+01 & 6.3e+01 \\
\hline                                                 
20 & $(40, 20)$ & 3.7e+01 & 5.4e+01 & 5.1e+01 & 7.0e+01 \\
\hline                                                 
24 & $(48, 24)$ & 3.8e+01 & 5.1e+01 & 4.4e+01 & 7.2e+01 \\
\hline                                                 
\end{tabular}                                          
\caption{Maximum effectivities for the American put with the Black-Scholes model at different time steps $n=5,10,15,20$ and for different error measures $E(\mu)$ and $N_{\max}=25$ in Algorithm~\ref{alg:global-greedy}. } \label{Table:BS}                  
\end{table}     

\begin{table}[ht]
\centering
The choice of $E(\mu)=E_{energy}^{Apost}(\mu)$\\[4pt]
\begin{tabular}{|c|c|c|c|c|c|}                       
\hline                                             
  $N_{\max}$& $(N_V, N_W)$ & $n = 5$ & $n = 10$ & $n = 15$ & $n = 20$ \\ 
\hline    \hline                                                 
4 & $(8, 4)$ & 4.0e+01 & 6.0e+01 & 7.1e+01 & 7.6e+01 \\ 
\hline                                                 
8 & $(16, 8)$ & 6.6e+01 & 1.1e+02 & 1.5e+02 & 1.5e+02 \\ 
\hline                                                 
16 & $(32, 16)$ & 1.2e+02 & 1.5e+02 & 1.7e+02 & 1.7e+02 \\
\hline                                                 
20 & $(40, 20)$ & 1.1e+02 & 1.1e+02 & 1.5e+02 & 1.9e+02 \\
\hline                                                 
24 & $(48, 24)$ & 1.9e+02 & 2.4e+02 & 3.8e+02 & 3.6e+02 \\
\hline                                            
\end{tabular} 
\\[8pt]
The choice of $E(\mu)=E_{energy}^{true}(\mu)$\\[4pt]
\centering
\begin{tabular}{|c|c|c|c|c|c|}                         
\hline                                                 
  $N_{\max}$ & $(N_V, N_W)$ & $n = 5$ & $n = 10$ & $n = 15$ & $n = 20$ \\   
\hline                                                 
4 & $(8, 4)$ & 4.2e+01 & 6.1e+01 & 7.6e+01 & 8.8e+01 \\ 
\hline                                                 
8 & $(16, 8)$ & 4.9e+01 & 8.7e+01 & 1.1e+02 & 1.2e+02 \\ 
\hline                                                 
16 & $(32, 16)$ & 8.6e+01 & 1.1e+02 & 1.1e+02 & 1.2e+02 \\
\hline                                                 
20 & $(40, 20)$ & 1.4e+02 & 1.5e+02 & 1.5e+02 & 1.7e+02 \\
\hline                                                 
24 & $(48, 24)$ & 1.4e+02 & 1.9e+02 & 2.0e+02 & 2.0e+02 \\
\hline                                                 
\end{tabular} 
\caption{Maximum effectivities for the American put with the Heston model at different time steps $n=5,10,15,20$ and for different error measures $E(\mu)$ and $N_{\max}=25$ in Algorithm~\ref{alg:global-greedy}. }\label{Table:Heston}
\end{table}

\end{document}